\documentclass[11pt,twoside]{article}
\usepackage[utf8]{inputenc}
\usepackage[english]{babel}
\usepackage{fancyhdr}
\usepackage{amssymb,amsmath,amsthm}
\usepackage{blindtext}
\usepackage{hyperref}
\usepackage{caption}
\usepackage[square,numbers]{natbib}

\usepackage{lipsum}
\usepackage{amsfonts}
\usepackage{graphicx}
\usepackage{epstopdf}
\usepackage{algorithmic}
\usepackage{algorithm}
\usepackage{mathtools}
\usepackage{enumerate}
\usepackage{mathrsfs}
\usepackage{xcolor}
\usepackage[inline]{enumitem}
\usepackage[titletoc,toc,title]{appendix}

\def\R{\mathbb R} % real numbers
\def\F{\mathcal F} % feasible set
\def\LS{\mathcal L^0} % level set
\def\Rzero{R^0}
\def\Lmaxj{\hat L^{\text{max}}}
\def\Dminstar{D^{\text{min}}(x^*)}
\def\Dmaxstar{D^{\text{max}}(x^*)}
\def\Gstar{G^*}
\def\kA{\bar k}
\def\kN{\hat k}
\def\kj{k^j}
\def\kz{k^z}

\def\Astar{\mathscr Z}
\def\neighj{\dfrac{\tau}{\tau+1} \Dmaxstar}
\def\neighA{\dfrac{\zeta(x^*)}{2L + \maxAL}}
\def\maxAL{\max\biggl\{\dfrac1{A_l},\dfrac{L^{\text{max}}}{2(1-\gamma)}\biggr\}}
\def\minAL{\min\biggl\{A_l,\dfrac{2(1-\gamma)}{L^{\text{max}}}\biggr\}}
\def\T{\max\biggl\{\dfrac1{A_l},\dfrac{L^{\text{max}}}{2\delta(1-\gamma)}\biggr\}}
\def\fdec{\biggl[\biggl(\T + 2 \Lmaxj\biggr) \Rzero + \Gstar\biggr]}
\def\fdecden{\biggl[\biggl(\T + 2 \Lmaxj\biggr) \Rzero + 2\Gstar\biggr]}
\def\rate{\sqrt{n-1}\max\biggl\{\frac{3A_u \sqrt{n-1}}{2\gamma},\frac1{L^{\text{max}}}\biggr\} \fdecden^2}

\DeclarePairedDelimiter{\norm}{\lVert}{\rVert} % norm
\DeclarePairedDelimiter{\abs}{\lvert}{\rvert} % absolute value
\DeclarePairedDelimiter{\innprod}{\langle}{\rangle}
\DeclarePairedDelimiter{\seminormj}{\lVert}{\rVert_{\innprod j}}
\DeclarePairedDelimiter{\seminormjk}{\lVert}{\rVert_{\innprod{j(k)}}}
\newtheorem{assumption}{Assumption}
\newtheorem{definition}{Definition}
\newtheorem{lemma}{Lemma}
\newtheorem{proposition}{Proposition}
\newtheorem{theorem}{Theorem}
\newtheorem{corollary}{Corollary}
\newtheorem{remark}{Remark}
\newtheorem{example}{Example}

\hypersetup{
    colorlinks=false,
    urlcolor=black,
    pdfborder={0 0 0}
}

\newcommand{\email}[1]{E-mail: \href{mailto:#1}{\texttt{#1}}}

\begin{document}
\thispagestyle{plain}

\setcounter{page}{1}

{\centering
%-----------------------%
% INSERT HERE THE TITLE %
%-----------------------%
{\LARGE \bfseries Active-set identification with complexity guarantees of an almost cyclic 2-coordinate descent method with Armijo line search}

\bigskip\bigskip
%-------------------------%
% INSERT HERE THE AUTHORS %
%-------------------------%
Andrea Cristofari$^*$
\bigskip

}

%----------------------------------%
% INSERT HERE AUTHORS' INFORMATION %
%----------------------------------%
\begin{center}
\small{\noindent$^*$Department of Mathematics ``Tullio Levi-Civita'' \\
University of Padua \\
Via Trieste, 63, 35121 Padua, Italy \\
\email{andrea.cristofari@unipd.it} \\
}
\end{center}

\bigskip\par\bigskip\par
\noindent \textbf{Abstract.}
In this paper, it is established finite active-set identification of an almost cyclic 2-coordinate descent method
for problems with one linear coupling constraint and simple bounds.
First, general active-set identification results are stated for non-convex objective functions.
Then, under convexity and a quadratic growth condition (satisfied by any strongly convex function),
complexity results on the number of iterations required to identify the active set are given.
In our analysis, a simple Armijo line search is used to compute the stepsize,
thus not requiring exact minimizations or additional information.

\bigskip\par
\noindent \textbf{Keywords.} Active-set identification. Surface identification. Manifold identification. Active-set complexity. Block coordinate descent methods.

\bigskip\par
\noindent \textbf{MSC2000 subject classifications.} 90C06. 90C30. 65K05.

\section{Introduction}\label{sec:intro}
In many different contexts, a desirable property of an optimization algorithm is the ability to identify,
in a finite number of iterations, a surface containing an optimal solution, in the sense that the points generated
by the algorithm eventually remain on that surface.
After such an identification, convergence can indeed be faster
since the algorithm can work in a lower dimensional space and, under proper assumptions,
it may also be possible to switch to methods with higher convergence rate.
Furthermore, in certain problems one may only be interested in knowing the structure of an optimal solution,
which can be revealed by identifying a surface where it lies,
without the need of running the algorithm to convergence
(for example, in lasso problems sparse solutions are promoted by the $\ell_1$ norm and
one may only be interested in knowing the support of an optimal solution).

In the literature, much effort has been devoted to proving identification properties of some algorithms for
smooth optimization~\cite{bertsekas:1976,bomze:2019a,bomze:2019b,burke:1990,burke:1988,burke:1994,calamai:1987,clarkson:2010,dunn:1987,gafni:1984,hare:2009,she:2017,wright:1993},
non-smooth optimization~\cite{daniilidis:2009,hare:2011,hare:2004,lewis:2011,liang:2017,mifflin:2002,nutini:2017,nutini:2019,sun:2019,wright:2012},
stochastic optimization~\cite{duchi:2016,lee:2012,poon:2018}
and derivative-free optimization~\cite{lewis:2010}.
Moreover, a wide class of methods, known as \textit{active-set methods}, has been object of extensive study from decades
(see, e.g.,~\cite{birgin:2002,cristofari:2017,cristofari:2017simplex,desantis:2016,facchinei:1998,hager:2006} and the references therein),
making use of specific techniques to identify the so called \textit{active set},
which is the set of constraints or variables that parametrizes a surface containing a solution.

The scope of the present paper is establishing finite active-set identification of a \mbox{2-coordinate} descent method,
proposed by the author in~\cite{cristofari:2019}, for smooth minimization problems with one linear equality constraint and simple bounds on the variables.
The main contributions of this paper can be summarized in the following points:
\begin{enumerate}[label=(\roman*), leftmargin=*]
\item The problem we consider here is not separable, due to a coupling constraint,
    and the method under analysis does not require first-order information to choose the working set,
    while guaranteeing deterministic convergence properties.

    These features represent major differences with the analysis of other block coordinate descent methods
    for which active-set identification results have been proved~\cite{cristofari:2019total,desantis:2016,hsieh:2008,luo:1992,luo:1993,nutini:2017,she:2017,wright:2012},
    since these methods either
    solve unconstrained problems where the objective function is the sum of a smooth term and a convex separable term
    (the latter might be an indicator function that enforces bound constraints),
    or allow for a non-separable structure but require full gradient evaluations to choose the working set,
    or have convergence results in expectation.
    In particular, active-set identification results are given in~\cite{she:2017} for variants of the sequential minimal optimization algorithm
    applied to the Support Vector Machine problem, where the authors consider a random selection of the working-set,
    which therefore does not require first-order information, but leads to convergence results in expectation.
\item Besides stating finite active-set identification results in a general non-convex setting,
    complexity results are also given under convexity of the objective function and a quadratic growth condition
    (satisfied by any strongly convex function), allowing us to bound the maximum number of iterations needed to identify the active set.
\end{enumerate}

Let us also remark that here we consider a simple Armijo line search for computing the stepsize along any search direction,
thus not requiring exact minimizations, or the knowledge of the Lipschitz constant of the gradient, or other additional information.
This makes our analysis of particular interest for realistic application to large-scale optimization problems.

\section{Preliminaries and Notation}\label{sec:preliminaries}
Let us first introduce part of the notation used in the paper.
Given a function $f \colon \R^n \to \R$, we indicate the gradient of $f$ by $\nabla f$
and we denote by $\nabla_i f$ its $i$th component (i.e., the $i$th partial derivative of $f$).
For a vector $x \in \R^n$, we denote by $x_i$ the $i$th component of $x$, we indicate by $\norm x$ the Euclidean norm of $x$
and we indicate by $\norm x_{\infty}$ the sup-norm of $x$.
We also denote by $e \in \R^n$ the vector made of all ones,
and by $e_i \in \R^n$ the vector that has the $i$th component equal to $1$ and all other components equal to $0$.
Given a scalar $a$, we indicate with $\lfloor a \rfloor$ the largest integer less than or equal to $a$.

Our analysis is concerned with the following problem:
\begin{equation}\label{prob}
\begin{split}
& \min \, f(x) \\
& e^T x = b \\
& l_i \le x_i \le u_i, \quad i = 1,\ldots,n,
\end{split}
\end{equation}
where $f \colon \R^n \to \R$ is a function with Lipschitz continuous gradient, $n \ge 2$, $b \in \R$
and, for all $i = 1,\ldots,n$, we have $l_i < u_i$, $l_i \in \R \cup \{-\infty\}$, $u_i \in \R \cup \{+\infty\}$.
The feasible set of problem~\eqref{prob} is denoted by $\F$.

Note that we may consider, instead of $e^T x = b$, any constraint of the form $a^T x = b$, with $a_i \ne 0$, $i = 1,\ldots,n$.
In such a case, problem~\eqref{prob} can be obtained by applying the variable transformation $x_i \leftarrow a_i \, x_i$ and setting the lower and the upper bound accordingly.
(Examples of relevant applications where problem~\eqref{prob} arises can be found, e.g., in~\cite{cristofari:2019} and the references therein.)

%Problems of the form of~\eqref{prob} arise in many different contexts: for example the training of Support Vector Machines,
%resource allocation and in all those applications requiring a minimization over the unit simplex (see, e.g.,~\cite{bomze:2002,deklerk:2008}).
%Moreover, any minimization problem over the convex set of given points can be rewritten as a minimization problem over the unit simplex
%(considering as variables the weights of the convex combination), and then, as in~\eqref{prob}.

The Lipschitz constant of $\nabla f$ over $\R^n$ is denoted by $L$, that is,
\[
\norm{\nabla f(x) - \nabla f(y)} \le L \norm{x-y}, \quad \forall \, x,y \in \R^n.
\]
It is possible to show~\cite{beck:2014} that there exist local Lipschitz constants
\begin{equation}\label{local_lipschitz}
L_{i,j} \le 2 L, \quad i,j = 1,\ldots,n,
\end{equation}
such that, for any $x \in \R^n$,
\[
\abs{\nabla f(x + s(e_i-e_j))^T (e_i-e_j) - \nabla f(x + t(e_i-e_j))^T (e_i-e_j)} \le L_{i,j} \abs{s-t}, \quad \forall \, s,t \in \R.
\]
Equivalently, defining $\phi_{i,j,x}(\alpha) = f(x+\alpha(e_i-e_j))$ and denoting its derivative by $\dot \phi_{i,j,x}$,
we have that
\begin{equation}\label{local_lips_phi}
\abs{\dot \phi_{i,j,x}(s) - \dot \phi_{i,j,x}(t)} \le L_{i,j} \abs{s-t}, \quad \forall \, s,t \in \R,
\end{equation}
that is, each derivative $\dot \phi_{i,j,x}$ is Lipschitz continuous over $\R$ with constant $L_{i,j}$.

Without loss of generality, we assume that all $L_{i,j}>0$, $i \ne j$
(if some of them are equal to zero, they can be replaced by positive overestimates)
and that $L_{i,i} = 0$, $i = 1,\ldots,n$.
We also define the following constants:
\begin{align}
& L^{\text{max}} = \max_{i,j = 1,\ldots,n} L_{i,j}, \label{Lmax} \\
& L_j = \sum_{i=1}^n L_{i,j}, \quad j = 1,\ldots,n, \label{Lj} \\
& \Lmaxj = \max_{j=1,\ldots,n} L_j. \label{Lmaxj}
\end{align}

A characterization of stationary points for problem~\eqref{prob} follows from KKT conditions.
In particular, a point $x^* \in \F$ is stationary for problem~\eqref{prob} if and only if there exists $\lambda^* \in \R$ such that,
for all $i = 1,\ldots,n$,
\begin{equation}\label{stat}
\nabla_i f(x^*)
\begin{cases}
\ge \lambda^*, \quad & \text{if } x^*_i = l_i, \\
= \lambda^*,   \quad & \text{if } x^*_i \in (l_i,u_i), \\
\le \lambda^*, \quad & \text{if } x^*_i = u_i.
\end{cases}
\end{equation}
Moreover, a variable $x^*_i \in \{l_i, u_i\}$ is said to satisfy the strict complementarity if $\nabla_i f(x^*) \ne \lambda^*$.
We also say that $x^*$ is non-degenerate if all variables $x^*_i$ such that $x^*_i \in \{l_i, u_i\}$ satisfy the strict complementarity.

In the following, we will make use of a simple operator between vectors in $\R^n$,
obtained from the usual dot product by discarding a certain component.
More precisely, for any $j \in \{1,\ldots,n\}$ we define the following positive semidefinite inner product:
\[
\innprod{x,y}_j = \sum_{i \ne j} x_i y_i, \quad \forall \, x,y \in \R^n.
\]
We also define the following seminorm, induced by the above inner product:
\[
\seminormj{x} = \sqrt{\innprod{x,x}_j}, \quad \forall \, x \in \R^n.
\]
Note that, by Cauchy-Bunyakovsky-Schwarz inequality, we have
\begin{equation}\label{cauchy}
\innprod{x,y}_j \le \norm x_{\innprod j} \, \norm y_{\innprod j}, \quad \forall \, x,y \in \R^n.
\end{equation}
In particular,~\eqref{cauchy} implies that
\begin{align}
\abs{x_i} & \le \seminormj{x}, \quad i \ne j, \hspace*{-88pt} & \forall \, x \in \R^n, \label{abs_vs_seminorm} \\
\sum_{i \ne j} \abs{x_i} & \le \sqrt{n-1} \seminormj{x}, \hspace*{-88pt} & \forall \, x \in \R^n. \label{sum_abs_vs_seminorm}
\end{align}
Moreover, it is straightforward to verify that
\begin{equation}\label{seminorm_le_norm}
\seminormj{x} \le \norm{x}, \quad \forall \, x \in \R^n.
\end{equation}

\section{Review of the algorithm}
Let us briefly review the algorithm proposed in~\cite{cristofari:2019},
named Almost Cyclic 2-Coordinate Descent (AC2CD) method, to solve problem~\eqref{prob}.
The main feature of AC2CD is an \textit{almost cyclic} rule to choose the working set.
This rule iteratively selects two variables: one is picked in a cyclic fashion,
while the other one is chosen by considering the distance from the bounds
in some points produced by the algorithm and remains in the working set until all the other variables have been picked.
Note the difference from the so-called \textit{essentially cyclic} rule, where all blocks of variables must be selected at least once
within a certain number of steps.

More precisely, at the beginning of each outer iteration $k$ of AC2CD we have a feasible point $x^k$ and we select a variable index $j(k)$ such that
$x^k_{j(k)}$ is ``sufficiently far'' from its nearest bound.
Then, we set the point $z^{k,1} = x^k$ and start a cycle of inner iterations, which are denoted by $(k,1),\ldots,(k,n)$.
In each inner iteration $(k,i)$,
we choose a working set of two variables: one of them is selected in a cyclic fashion, while the other one remains the $j(k)$th variable.
So, we produce a feasible point $z^{k,i+1}$ from $z^{k,i}$ by moving only the two variables in the working set.
At the end of the last inner iteration we finally set $x^{k+1} = z^{k,n+1}$ and start a new outer iteration $k+1$.

Let us remark that our algorithm does not use first-order information to choose the working set. Moreover, as to be described later,
only two partial derivatives are required to move each pair of variables.
We can hence achieve high computational efficiency if partial derivative evaluation for the objective function is much cheaper than full gradient evaluation.
For instance, this is the case when $f$ is the sum of univariate functions
(such as in the problems considered in~\cite{necoara:2017} for large-scale network optimization).
Other interesting examples, including the Support Vector Machine problem and the Chebyshev center problems,
are those where the objective function is quadratic of the form $f(x) = x^T Q^T Q x - q^T x$, with $Q$ being a given $m \times n$ matrix and $q$ being a given vector.
In this case, a partial derivative of $f(x)$ can be computed with a cost $\mathcal O(m)$, while computing the whole gradient has a cost $\mathcal O(mn)$
(see~\cite{cristofari:2019} for details).
%\item As already observed before, our almost cyclic rule allows us not to cycle through all the pairs of variables
%    (whose number grows quadratically with $n$).
%\end{itemize}

Now, let us explain in more detail how the index $j(k)$ is chosen at the beginning of an outer iteration $k$
and how the two variables in the working set are moved in the inner iterations $(k,1),\ldots,(k,n)$.

For what concerns the choice of $j(k)$, for any $x \in \F$ let us first define
\begin{equation}\label{Dh}
D_h(x) = \min\{x_h-l_h,u_h-x_h\}, \quad h = 1,\ldots,n.
\end{equation}
Namely, $D_h(x)$ returns the distance of $x_h$ from its nearest bound.
Moreover, for any point $x^k$ produced by the algorithm, we define $D^k$ as
the maximum distance between each component of $x^k$ and its nearest bound, that is,
\begin{equation}\label{Dk}
D^k = \max_{h=1,\ldots,n} D_h(x^k).
\end{equation}
Then, $j(k)$ can be chosen as any index satisfying
\begin{equation}\label{j(k)}
D_{j(k)} (x^k) \ge \tau D^k,
\end{equation}
where $\tau \in (0,1]$ is a fixed parameter.
In other words, the distance between $x^k_{j(k)}$ and its nearest bound must be sufficiently large compared to $D^k$.

For what concerns the variable update,
let us denote by $p^k_i$ the variable index that is selected in a cyclic manner at an inner iteration $(k,i)$
(note that the variables can be taken in any order).
So, $z^{k,i}_{p^k_i}$ and $z^{k,i}_{j(k)}$ are the two variables that can be moved from $z^{k,i}$.
To do this, we use the following search direction (which has at most two non-zero components and maintains feasibility for the equality constraint):
\begin{equation}\label{g_d_def}
d^{k,i} = g^{k,i} (e_{p^k_i} - e_{j(k)}), \quad \text{where} \quad g^{k,i} = \nabla_{j(k)} f(z^{k,i}) - \nabla_{p^k_i} f(z^{k,i}),
\end{equation}
and we set
\[
z^{k,i+1} = z^{k,i} + \alpha^{k,i} d^{k,i},
\]
where $\alpha^{k,i}$ is a suitably computed feasible stepsize.
Note that
\begin{equation}\label{ac2cd_prop_2}
\nabla f(z^{k,i})^T d^{k,i} = -(g^{k,i})^2,
\end{equation}
and then, every non-zero $d^{k,i}$ is a descent direction.
The scheme of AC2CD is reported in Algorithm~\ref{alg:ac2cd}.

\begin{algorithm}[h!]
\caption{\texttt{\hspace*{0.1truecm}Almost Cyclic 2-Coordinate Descent (AC2CD) method}}
\label{alg:ac2cd}
\begin{algorithmic}
\item[]\hspace*{-0.1truecm}$\,\,\,0$\hspace*{0.1truecm}
    \textbf{Given} $x^0 \in \F$ and $\tau \in (0,1]$
\item[]\hspace*{-0.1truecm}$\,\,\,1$\hspace*{0.1truecm} \textbf{For} $k = 0,1,\ldots$
\item[]\hspace*{-0.1truecm}$\,\,\,2$\hspace*{0.6truecm} Choose a variable index $j(k) \in \{1,\ldots,n\}$ that satisfies~\eqref{j(k)}
\item[]\hspace*{-0.1truecm}$\,\,\,3$\hspace*{0.6truecm} Choose a permutation $\{p^k_1,\ldots,p^k_n\}$ of $\{1,\ldots,n\}$
\item[]\hspace*{-0.1truecm}$\,\,\,4$\hspace*{0.6truecm} Set $z^{k,1} = x^k$
\item[]\hspace*{-0.1truecm}$\,\,\,5$\hspace*{0.6truecm} \textbf{For} $i = 1,\ldots,n$
\item[]\hspace*{-0.1truecm}$\,\,\,6$\hspace*{1.1truecm} Let $g^{k,i} = \nabla_{j(k)} f(z^{k,i}) - \nabla_{p^k_i} f(z^{k,i})$
\item[]\hspace*{-0.1truecm}$\,\,\,7$\hspace*{1.1truecm} Compute the search direction $d^{k,i} = g^{k,i} (e_{p^k_i} - e_{j(k)})$
\item[]\hspace*{-0.1truecm}$\,\,\,8$\hspace*{1.1truecm} Compute a feasible stepsize $\alpha^{k,i}$ and set $z^{k,i+1} = z^{k,i} + \alpha^{k,i} d^{k,i}$
\item[]\hspace*{-0.1truecm}$\,\,\,9$\hspace*{0.6truecm} \textbf{End for}
\item[]\hspace*{-0.1truecm}$10$\hspace*{0.6truecm} Set $x^{k+1} = z^{k,n+1}$
\item[]\hspace*{-0.1truecm}$11$\hspace*{0.1truecm} \textbf{End for}
\end{algorithmic}
\end{algorithm}

\subsection{Computation of the stepsize}
Under a technical assumption (see Assumption~\ref{assumpt:l0_int_point} in the next section),
global convergence of AC2CD to stationary points was established in~\cite{cristofari:2019}
for different choices of the stepsize $\alpha^{k,i}$ (to be used at line~$8$ of Algorithm~\ref{alg:ac2cd}),
including the Armijo stepsize, overestimates of the local Lipschitz constants of $\nabla f$ and the exact stepsize for strictly convex objective
functions\footnote{For general conditions on the stepsize,
see SC (Stepsize Condition)~1 in~\cite{cristofari:2019}.
A typo is present in point~(i) of SC~1 in~\cite{cristofari:2019}: $f(z^{k,i+i})$ should be replaced by $f(z^{k,i+1})$.}.

Here we focus on the case where, at every inner iteration $(k,i)$,
the stepsize $\alpha^{k,i}$ is computed by the Armijo line search,
which is a backtracking procedure that computes a stepsize in a finite number of iterations.
The scheme of the Armijo line search used in AC2CD is reported in Algorithm~\ref{alg:armijo}.

\begin{algorithm}[h!]
\caption{\texttt{\hspace*{0.1truecm}Armijo line search} (to compute $\alpha^{k,i}$ at step~8 of AC2CD)}
\label{alg:armijo}
\begin{algorithmic}
\item[]\hspace*{-0.1truecm}$\,\,\,0$\hspace*{0.1truecm}
    \textbf{Given} the search direction $d^{k,i}$ and two parameters $\gamma \in (0,1)$, $\delta \in (0,1)$
\item[]\hspace*{-0.1truecm}$\,\,\,1$\hspace*{0.1truecm} Choose a feasible stepsize $\Delta^{k,i} \ge 0$ and set $\alpha = \Delta^{k,i}$
\item[]\hspace*{-0.1truecm}$\,\,\,2$\hspace*{0.1truecm} \textbf{While} $f(z^{k,i} + \alpha d^{k,i}) > f (z^{k,i}) + \gamma \alpha \nabla f(z^{k,i})^Td^{k,i}$
\item[]\hspace*{-0.1truecm}$\,\,\,3$\hspace*{0.6truecm} Set $\alpha = \delta \alpha$
\item[]\hspace*{-0.1truecm}$\,\,\,4$\hspace*{0.1truecm} \textbf{End while}
\item[]\hspace*{-0.1truecm}$\,\,\,5$\hspace*{0.1truecm} Return $\alpha^{k,i} = \alpha$
\end{algorithmic}
\end{algorithm}

We see that the considered Armijo line search is very simple and does not require exact minimizations or
additional information (such as the knowledge of the Lipschitz constant of $\nabla f$).
For this reason, it can be an effective choice for non-convex large-scale problems and when no closed form is known for the stepsize.

To obtain global convergence of AC2CD to stationary points, an appropriate choice of the initial stepsize $\Delta^{k,i}$
at line~$1$ of Algorithm~\ref{alg:armijo} is needed.
In~\cite{cristofari:2019} it was shown that, at every inner iteration $(k,i)$, a possible choice is the following:
\begin{equation}\label{initial_stepsize_armijo_1}
\Delta^{k,i} = \min\{\bar \alpha^{k,i}, A^{k,i}\},
\end{equation}
where
\begin{itemize}
\item $\bar \alpha^{k,i}$ is the largest feasible stepsize along the direction $d^{k,i}$, that is,
    \begin{equation}\label{max_stepsize_def}
    \bar \alpha^{k,i} =
    \begin{cases}
    \dfrac1{g^{k,i}} \, \min \{u_{p^k_i}-z^{k,i}_{p^k_i},z^{k,i}_{j(k)}-l_{j(k)}\}, \quad       & \text{if } g^{k,i} > 0, \\[1.5ex]
    \dfrac1{\abs{g^{k,i}}} \, \min \{z^{k,i}_{p^k_i}-l_{p^k_i},u_{j(k)}-z^{k,i}_{j(k)}\}, \quad & \text{if } g^{k,i} < 0, \\[1.5ex]
    0,                                                                                     & \text{if } g^{k,i} = 0;
    \end{cases}
    \end{equation}
\item $A^{k,i}$ must be chosen between two finite positive constants, that is,
    \begin{equation}\label{initial_stepsize_armijo_2}
    0 < A_l \le A^{k,i} \le A_u < \infty,
    \end{equation}
    with $A_l$ and $A_u$ being two fixed parameters.
\end{itemize}
We observe that, in~\eqref{max_stepsize_def}, we set $\bar \alpha^{k,i}=0$
when $g^{k,i} = 0$, i.e., when $d^{k,i} = 0$ (see~\eqref{g_d_def}).
Therefore, $\bar \alpha^{k,i}$ is not actually the largest feasible stepsize along $d^{k,i}$ when $d^{k,i} = 0$.
This choice in the definition of $\bar \alpha^{k,i}$ simplifies the analysis and entails no loss of generality,
since it stills guarantees that $z^{k,i+1} = z^{k,i}$ when $d^{k,i} = 0$.
In particular, note that
\begin{equation}\label{ac2cd_prop_3}
d^{k,i} = 0 \quad \overset{\text{\eqref{g_d_def}}}\Leftrightarrow \quad g^{k,i}=0
\quad \overset{\text{\eqref{max_stepsize_def}}}\Rightarrow \quad \bar \alpha^{k,i} = 0
\quad \Leftrightarrow \quad z^{k,i+1} = z^{k,i}.
\end{equation}
To obtain the last relation in~\eqref{ac2cd_prop_3},
we can use~\eqref{ac2cd_prop_2}, \eqref{initial_stepsize_armijo_1} and~\eqref{initial_stepsize_armijo_2}, leading to
\[
\bar \alpha^{k,i} > 0 \quad \Leftrightarrow \quad \Delta^{k,i} > 0 \, \land \, \nabla f(z^{k,i})^T d^{k,i} < 0.
\]
So, if $\bar \alpha^{k,i} > 0$, the Armijo line search returns a stepsize $\alpha^{k,i} > 0$, implying that $z^{k,i+1} \ne z^{k,i}$.
Vice versa, if $\bar \alpha^{k,i} = 0$,
the Armijo line search returns $\alpha^{k,i} = 0$, implying that $z^{k,i+1} = z^{k,i}$.
Namely, the last relation in~\eqref{ac2cd_prop_3} holds.

\section{Basic assumptions}\label{sec:assumptions}
Let $X^*$ be the set of all stationary points for problem~\eqref{prob} and
also define the level set
\[
\LS = \{x \in \F \colon f(x) \le f(x^0)\},
\]
where $\F$ is the feasible set of problem~\eqref{prob} and $x^0$ is the starting point used in AC2CD.
We assume that $\LS$ is non-empty and compact
(implying that both the feasible set $\F$ and the set of stationary points $X^*$ are non-empty as well).

According to the results stated in~\cite{cristofari:2019},
we also need the following assumption on the level set $\LS$ to ensure global convergence of AC2CD
(in the sense that every limit point of the sequence $\{x^k\}$ produced by the algorithm is stationary):
\begin{assumption}\label{assumpt:l0_int_point}
$\forall \, x \in \LS, \, \exists \, i \in \{1,\ldots,n\} \colon x_i \in (l_i,u_i)$.
\end{assumption}
Namely, we require that every point of $\LS$ has at least one component strictly between the lower and the upper bound.
Note that Assumption~\ref{assumpt:l0_int_point} is automatically satisfied when $\F$ is the unit simplex
(i.e., when in problem~\eqref{prob} we have $b=1$, $l_i = 0$, $u_i = +\infty$, $i=1,\ldots,n$).
Moreover, in~\cite{lin:2009} it is shown that Assumption~\ref{assumpt:l0_int_point} is also satisfied for the Support Vector Machine training problem
if $f(x^0)<0$ and the smallest eigenvalue of the Hessian matrix of $f(x)$ is sufficiently large.
(Assumption~\ref{assumpt:l0_int_point} is satisfied also when at least one variable has are no finite bounds, provided $\F$ is not a singleton.)

Essentially, Assumption~\ref{assumpt:l0_int_point} is needed to prevent AC2CD from converging to a point $x^*$ with all components at the lower or the upper bound.
To be more specific, the convergence analysis of AC2CD (see~\cite{cristofari:2019}) relies on the fact that eventually $l_{j(k)} < x^k_{j(k)} < u_{j(k)}$
and that $\nabla_{j(k)} f(x^k)$ converges (over suitable subsequences) to the KKT multiplier $\lambda^*$ appearing in~\eqref{stat}.
Also the analysis of the active-set identification reported later uses the same properties
(see Proposition~\ref{prop:j_active_set} and the proof of Theorem~\ref{th:active_set_general_1} below).
Without Assumption~\ref{assumpt:l0_int_point}, all these results do not hold, since
$\{x^k\}$ may have limit points with all components at the lower or the upper bound.

We also observe that, for every outer iteration $k \ge 0$, Assumption~\ref{assumpt:l0_int_point} ensures that
$x^{k+1} \ne x^k$ if and only if $x^k$ is non-stationary.
To see this, under Assumption~\ref{assumpt:l0_int_point} observe that $l_{j(k)} < x^k_{j(k)} < u_{j(k)}$
for all $k \ge 0$ (since $j(k)$ must satisfy~\eqref{j(k)} with $D^k>0$).
Then, from the KKT conditions~\eqref{stat}, there exists a feasible descent direction in the inner iterations
$(k,1),\ldots,(k,n)$ if and only if $x^k$ is non-stationary.

On the contrary, without Assumption~\ref{assumpt:l0_int_point}, the algorithm may end up in a non-stationary point $x^k$ with all components at the lower or the upper bound.
In such a case, even if every choice of $j(k) = 1,\ldots,n$ satisfies~\eqref{j(k)} (since $D^k = D_h(x^k) = 0$, $h = 1,\ldots,n$),
for certain choices of $j(k)$ there may not exist a feasible descent direction in any inner iteration $(k,1),\ldots,(k,n)$.
Namely, AC2CD may get stuck in a non-stationary point $x^k$.
This issue can be overcome by introducing an anticycling rule to select $j(k)$ when such a point $x^k$ is produced.
Doing so, we may relax Assumption~\ref{assumpt:l0_int_point} by requiring only the stationary points in $\LS$
not to have all components at the lower or the upper bound, but in our analysis we use Assumption~\ref{assumpt:l0_int_point} for simplicity.

Overcoming the limitation deriving from Assumption~\ref{assumpt:l0_int_point} by properly modifying the algorithm
might be a challenging subject for future research.

In the rest of the paper, we will consider all the above assumptions
always satisfied, even if not explicitly invoked.
Namely, we will consider $\LS$ non-empty and compact and we will consider Assumption~\ref{assumpt:l0_int_point} satisfied.

\section{Technical results}\label{technical_results}
In this section, we fix a few concepts and give some technical results.
First note that, for every inner iteration $(k,i)$ of AC2CD,
\begin{equation}\label{ac2cd_prop_1}
p^k_i \ne j(k) \quad \Rightarrow \quad z^{k,i}_{p^k_i} = x^k_{p^k_i} \text{ and } x^{k+1}_{p^k_i} = z^{k,i+1}_{p^k_i},
\end{equation}
since each coordinate, except the $j(k)$th one, is moved (at most) once in a cycle of inner iterations.

Furthermore, there is a relation between the Armijo stepsize and the local Lipschitz constants of $\nabla f$:
at any inner iteration $(k,i)$, every stepsize $\alpha \le 2 (1-\gamma) / L_{p^k_i,j(k)}$ satisfies the so called Armijo condition,
which is the exit condition in the while loop of Algorithm~\ref{alg:armijo}.
Namely,
$f (z^{k,i} + \alpha d^{k,i}) \le f (z^{k,i}) + \gamma \alpha \nabla f(z^{k,i})^T d^{k,i}$
for all $\alpha \in [0, 2 (1-\gamma)/L_{p^k_i,j(k)}]$
(see the proof of Proposition~3 in~\cite{cristofari:2019}).
Since, in our line search, $\alpha$ is multiplied by $\delta \in (0,1)$ until the Armijo condition is satisfied
(see line~3 in Algorithm~\ref{alg:armijo}), we immediately have the following result.
\begin{lemma}\label{lemma:lower_bound_stepsize}
At every inner iteration $(k,i)$, the initial stepsize $\Delta^{k,i}$ used in the Armijo line search is such that
\begin{align*}
\Delta^{k,i} \le \frac{2 (1-\gamma)}{L_{p^k_i,j(k)}} \; & \Rightarrow \; \alpha^{k,i} = \Delta^{k,i}, \\
\Delta^{k,i} > \frac{2 (1-\gamma)}{L_{p^k_i,j(k)}} \; & \Rightarrow \; \alpha^{k,i} \in \biggl(\frac{2 \delta (1-\gamma)}{L_{p^k_i,j(k)}},\Delta^{k,i}\biggr],
\end{align*}
where, in the Armijo line search, $\gamma \in (0,1)$ is the parameter for sufficient decrease and
$\delta \in (0,1)$ is the reduction parameter.
Therefore, $\displaystyle{\alpha^{k,i} \ge \min\biggl\{\Delta^{k,i}, \frac{2 \delta (1-\gamma)}{L_{p^k_i,j(k)}}\biggr\}}$.
\end{lemma}

As a consequence of Lemma~1 in~\cite{cristofari:2019}, we also have the following relation between the limit of $\{x^k\}$
and the limit of the sequences $\{z^{k,i}\}$, $i = 1,\ldots,n$:
\begin{equation}\label{conv_z}
\lim_{k \to \infty} x^k = x^* \; \Leftrightarrow \; \lim_{k \to \infty} z^{k,i} = x^*, \quad i=1,\ldots,n.
\end{equation}

Now we state some useful properties derived from the semidefinite inner product and the seminorm defined at the end of Section~\ref{sec:preliminaries}.
In the following results, we use $L_j$ as defined in~\eqref{Lj}.
The proofs are reported in Appendix~\ref{app:technical results}.
\begin{lemma}\label{lemma:prod}
For any $j \in \{1,\ldots,n\}$ we have that
\[
v^T (x'-x'') = \innprod{v-v_j e, x'-x''}_j, \quad \forall \, x',x'' \in \F, \quad \forall \, v \in \R^n.
\]
\end{lemma}

\begin{lemma}\label{lemma:lips_const}
If $f$ is convex over $\R^n$, for any $j \in \{1,\ldots,n\}$ we have that
\[
\bigl\|[\nabla f(x') - \nabla_j f(x') e] - [\nabla f(x'') - \nabla_j f(x'') e]\bigr\|_{\innprod j}
\le L_j \seminormj{x'-x''}, \quad \forall \, x',x'' \in \F
\]
\end{lemma}

\begin{corollary}\label{corol:lips}
If $f$ is convex over $\R^n$, at every inner iteration $(k,i)$ of AC2CD we have that
\[
\abs{\nabla_{p^k_i} f(v)-\nabla_{j(k)} f(v)+g^{k,i}} \le L_{j(k)} \seminormjk{v-z^{k,i}}, \quad \forall \, v \in \R^n.
\]
\end{corollary}

\begin{lemma}\label{lemma:lips_descent}
If $f$ is convex over $\R^n$, for any $j \in \{1,\ldots,n\}$ we have that
\[
f(x'') \le f(x') + \nabla f(x')^T (x''-x') + \frac {L_j}2 \seminormj{x'-x''}^2, \quad \forall \, x',x'' \in \F.
\]
\end{lemma}

\section{Active-set identification in the non-convex case}\label{sec:active_set_general}
In this section, we show that AC2CD identifies the active set of problem~\eqref{prob} in a finite number of iterations,
without any assumption on the convexity of $f$.

First of all, let us give the definition of active set for our problem.
\begin{definition}
Given a stationary point $x^*$ of problem~\eqref{prob}, we define the \textit{active set} as
\[
\Astar(x^*) = \{i \colon x^*_i = l_i\} \cup  \{i \colon x^*_i = u_i\}.
\]
We also define
\[
\Astar^+(x^*) = \Astar(x^*) \cap \{i \colon \nabla_i f(x^*) \ne \lambda^*\},
\]
where $\lambda^*$ is the KKT multiplier associated with $x^*$ appearing in~\eqref{stat}.
\end{definition}

We see that $\Astar(x^*)$ is the set of indices of all the variables that are at the lower or the upper bound in a stationary point $x^*$,
whereas $\Astar^+(x^*)$ contains only the indices of the variables satisfying the strict complementarity.
We notice that, from a geometric perspective,
$\Astar^+(x^*)$ defines the face of $\F$ exposed to $-\nabla f(x^*)$~\cite{burke:1994}.

The scope of this section is two-fold:
\begin{enumerate}[label=(\roman*), leftmargin=*]
\item Firstly, it will be shown that, given a sequence of points $\{x^k\} \to x^*$ produced by AC2CD, an iteration $\kA$ exists such that, for all $k > \kA$,
    \begin{equation}\label{active_set_identification}
    x^k_h = x^*_h, \quad \forall \, h \in \Astar^+(x^*).
    \end{equation}
    Namely, in a finite number of iterations AC2CD sets to the bounds all the variables that satisfy the strict complementarity at $x^*$.
\item Secondly, we will give a characterization of the neighborhood of $x^*$ where~\eqref{active_set_identification} holds,
    which will be used in Section~\ref{sec:active_set_complexity} to obtain an upper bound for $\kA$
    (under convexity of $f$ and a quadratic growth condition).
\end{enumerate}

Note that, as common when analyzing active-set identification properties of an optimization algorithm,
here we require the whole sequence $\{x^k\}$ to converge.
For AC2CD, in~\cite{cristofari:2019} it was shown that every limit point of $\{x^k\}$ is stationary and,
if $\{f(x^k)\}$ converges, then $\lim_{k \to \infty} \norm{z^{k,i+1}-z^{k,i}} = 0$, $i=1,\ldots,n$,
implying that $\lim_{k \to \infty} \norm{x^{k+1}-x^k} = 0$ if a limit point of $\{x^k\}$ exists.
So, using the same arguments given in~\cite[Theorem~14.1.5]{ortega:1970},
we get that the whole sequence $\{x^k\}$ converges if the number of stationary points in $\LS$ is finite.
By a more general result stated in~\cite[Proposition~8.3.10]{facchinei:2003}, we also have that the whole sequence $\{x^k\}$
converges if it has an isolated limit point.
Other conditions can be obtained from~\cite[Theorem~4.3]{bomze:2019a}:
if $f$ satisfies a suitable descent property along the search directions,
then a strict local minimum with no other stationary points in its neighborhood attracts the whole sequence $\{x^k\}$.

Now, we start our analysis by giving an intermediate result stating that,
in a neighborhood of $x^*$, the index $j(k)$ is such that $l_{j(k)} < x^*_{j(k)} < u_{j(k)}$.

\begin{proposition}\label{prop:j_active_set}
Let $\{x^k\}$ be a sequence of points produced by AC2CD and assume that $\lim_{k \to \infty} x^k = x^*$.
Define the maximum distance from the bounds at $x^*$ as
\[
\Dmaxstar = \max_{i = 1,\ldots,n} D_i(x^*),
\]
which is positive by Assumption~\ref{assumpt:l0_int_point}, and let $\kj$ be the first outer iteration such that
\[
\norm{x^k-x^*}_{\infty} < \neighj, \quad \forall \, k \ge \kj,
\]
where $\tau \in (0,1]$ is the parameter used to choose $j(k)$, satisfying~\eqref{j(k)}.
Then, for all $k \ge \kj$ we have that $j(k) \notin \Astar(x^*)$.
\end{proposition}

\begin{proof}
Consider an outer iteration $k \ge \kj$ and let $\hat \jmath$ be an index such that
$D_{\hat \jmath}(x^*) = \Dmaxstar$. We have
$\abs{x^k_{\hat \jmath} - x^*_{\hat \jmath}} \le \norm{x^k-x^*}_{\infty} < \neighj$,
implying that
\begin{equation}\label{Dh_proof_2}
x^k_{\hat \jmath} - l_{\hat \jmath} > x^*_{\hat \jmath} - l_{\hat \jmath} - \frac{\tau}{\tau+1} D_{\hat \jmath}(x^*) \quad \text{and} \quad
u_{\hat \jmath} - x^k_{\hat \jmath} > u_{\hat \jmath} - x^*_{\hat \jmath} - \frac{\tau}{\tau+1} D_{\hat \jmath}(x^*).
\end{equation}
Therefore, we can write
\begin{equation}\label{Dh_proof_3}
\begin{split}
D_{\hat \jmath}(x^k) & = \min\{x^k_{\hat \jmath}-l_{\hat \jmath},u_{\hat \jmath}-x^k_{\hat \jmath}\}
\overset{\text{\eqref{Dh_proof_2}}} > \min\{x^*_{\hat \jmath}-l_{\hat \jmath},u_{\hat \jmath}-x^*_{\hat \jmath}\} - \frac{\tau}{\tau+1} D_{\hat \jmath}(x^*) \\
& = D_{\hat \jmath}(x^*) - \frac{\tau}{\tau+1} D_{\hat \jmath}(x^*) = \frac 1{\tau+1} D_{\hat \jmath}(x^*).
\end{split}
\end{equation}
Arguing by contradiction, assume now that $j(k) \in \Astar(x^*)$, that is,
\begin{equation}\label{j_contr_proof}
x^*_{j(k)} \in \{l_{j(k)}, u_{j(k)}\}.
\end{equation}
We obtain
\begin{equation*}
\begin{split}
D_{j(k)} (x^k) & = \min\{x^k_{j(k)}-l_{j(k)},u_{j(k)}-x^k_{j(k)}\}
                 \overset{\text{\eqref{j_contr_proof}}} \le \abs{x^k_{j(k)}-x^*_{j(k)}} \le \norm{x^k - x^*}_{\infty} \\
               & < \frac{\tau}{\tau+1} \Dmaxstar =
                 \frac{\tau}{\tau+1} D_{\hat \jmath}(x^*) \overset{\text{\eqref{Dh_proof_3}}} < \tau D_{\hat \jmath}(x^k)
                 \overset{\text{\eqref{Dk}}} \le \tau D^k,
\end{split}
\end{equation*}
contradicting~\eqref{j(k)}.
\end{proof}

Combining the above proposition with~\eqref{conv_z}, the next result immediately follows.
\begin{proposition}\label{prop:z_j}
Let $\{x^k\}$ be a sequence of points produced by AC2CD and assume that $\lim_{k \to \infty} x^k = x^*$.
There exists an iteration $\kz$ such that, for all $k \ge \kz$,
\[
l_{j(k)} < z^{k,i}_{j(k)} < u_{j(k)}, \quad i = 1,\ldots,n+1.
\]
\end{proposition}

Now, we are ready to show that~\eqref{active_set_identification} holds for all sufficiently large iterations.
Our analysis takes inspiration from the one in~\cite{nutini:2019} for proximal gradient methods,
where it is proved that the active set is identified in a neighborhood of the optimal solution under the non-degeneracy assumption.
That neighborhood is defined in~\cite{nutini:2019} by using a problem-dependent constant related on ``the amount of degeneracy'' of the optimal solution.

Here, for a stationary point $x^*$ such that $\Astar^+(x^*) \ne \emptyset$, we define the following positive constant,
measuring the ``minimum amount of strict complementarity'' at $x^*$:
\begin{equation}\label{nondeg_measure}
\zeta(x^*) = \min_{i \in \Astar^+(x^*)} \abs{\nabla_i f(x^*) - \lambda^*},
\end{equation}
where $\lambda^*$ is the KKT multiplier associated to $x^*$, according to~\eqref{stat}.

\begin{theorem}\label{th:active_set_general_1}
Let $\{x^k\}$ be a sequence of points produced by AC2CD and assume that $\lim_{k \to \infty} x^k = x^*$.
Let $\kA$ be the first outer iteration such that
\begin{equation}\label{neighborhood_A}
\norm{z^{k,i}-x^*} < \neighA, \quad i = 1,\ldots,n, \quad \forall \, k \ge \kA,
\end{equation}
where $\zeta(x^*)>0$ is the minimum strict complementarity measure at $x^*$, defined as in~\eqref{nondeg_measure},
$L$ is the Lipschitz constant of $\nabla f$,
$A_l>0$ is the lower bound on the parameter $A^{k,i}$ used to compute the initial stepsize $\Delta^{k,i}$
in the Armijo line search (see~\eqref{initial_stepsize_armijo_1} and~\eqref{initial_stepsize_armijo_2}),
$\gamma \in (0,1)$ is the parameter for sufficient decrease in the Armijo line search
and $L^{\text{max}}>0$ is the maximum among the local Lipschitz constants $L_{i,j}$,
defined as in~\eqref{Lmax}.

Also assume that $\kA \ge \max\{\kj,\kz\}$, where $\kj$ is the first outer iteration such that $j(k) \notin \Astar(x^*)$ for all $k \ge \kj$, defined as in
Proposition~\ref{prop:j_active_set}, and $\kz$ is the first outer iteration  such that
$l_{j(k)} < z^{k,i}_{j(k)} < u_{j(k)}$, $i = 1,\ldots,n+1$, for all $k \ge \kz$, defined as in Proposition~\ref{prop:z_j}.

Then, for all $k > \kA$ we have that
\[
x^k_h = x^*_h, \quad \forall \, h \in \Astar^+(x^*).
\]
\end{theorem}

\begin{proof}
Consider an outer iteration $k \ge \kA$ and any index $h \in \Astar^+(x^*)$.
Moreover, let $(k,i)$ be the inner iteration where $p^k_i = h$.
Without loss of generality, let us assume that $x^*_h = l_h$ (the proof for the case where $x^*_h = u_h$ is analogous). Namely,
\begin{equation}\label{xstar_l}
x^*_h = l_h \quad \text{and} \quad \nabla_h f(x^*) > \lambda^*,
\end{equation}
where $\lambda^*$ is the KKT multiplier associated to $x^*$, according to the stationary conditions~\eqref{stat}.
Since $k \ge \kj$, from Proposition~\ref{prop:j_active_set} we have that
\begin{equation}\label{j(k)_N}
j(k) \notin \Astar(x^*),
\end{equation}
implying that $h \ne j(k)$.
Then, using~\eqref{j(k)_N} and the stationary conditions~\eqref{stat}, we get $\lambda^* = \nabla_{j(k)} f(x^*)$.
Recalling the definition of $\zeta (x^*)$, it follows that
\[
\zeta(x^*) \le \nabla_h f(x^*) - \nabla_{j(k)} f(x^*).
\]
Moreover, from the definition of $g^{k,i}$ given in~\eqref{g_d_def} we can write
\begin{equation*}
\begin{split}
\nabla_h f(x^*) - \nabla_{j(k)} f(x^*) + g^{k,i} & = \nabla_h f(x^*) - \nabla_{j(k)} f(x^*) + \nabla_{j(k)} f(z^{k,i}) - \nabla_h f(z^{k,i}) \\
                                                 & \le \abs{\nabla_h f(x^*) - \nabla_h f(z^{k,i})} + \abs{\nabla_{j(k)} f(z^{k,i}) - \nabla_{j(k)} f(x^*)} \\
                                                 & \le 2 \norm{\nabla f(x^*) - \nabla f(z^{k,i})} \le 2 L \norm{x^*-z^{k,i}},
\end{split}
\end{equation*}
and then,
\begin{equation}\label{zeta_proof}
\zeta(x^*) \le -g^{k,i} + 2 L \norm{x^*-z^{k,i}}.
\end{equation}
Now, we can rewrite~\eqref{neighborhood_A} by multiplying the numerator and the denominator of the right-hand side by
$\maxAL^{-1} = \minAL$, obtaining
\[
\norm{z^{k,i}-x^*} < \dfrac{\zeta(x^*)\minAL}{2L\minAL + 1}.
\]
Multiplying both sides of this inequality by the denominator of the right-hand side, we can write
\begin{equation}
\begin{split}\label{zki-xstar_proof}
\norm{z^{k,i}-x^*} & \overset{\phantom{\text{\eqref{zeta_proof}}}} = (\zeta(x^*)-2L\norm{z^{k,i}-x^*}) \minAL \\
                   & \overset{\text{\eqref{zeta_proof}}} \le -g^{k,i} \minAL.
\end{split}
\end{equation}
It follows that $g^{k,i} \le 0$. If $g^{k,i} = 0$, we have
\[
x^{k+1}_h \overset{\text{\eqref{ac2cd_prop_1}}} = z^{k,i+1}_h \overset{\text{\eqref{g_d_def}}} = z^{k,i}_h \overset{\text{\eqref{zki-xstar_proof}}} = x^*_h,
\]
and the desired result is thus obtained.
Now assume that $g^{k,i} < 0$. We can upper bound the largest feasible stepsize $\bar \alpha^{k,i}$ as follows:
\begin{equation}\label{max_stepsize_proof}
\begin{split}
\bar \alpha^{k,i} & \overset{\text{\eqref{max_stepsize_def}}} \le -\frac{z^{k,i}_h-l_h}{g^{k,i}}
\overset{\text{\eqref{xstar_l}}} = -\frac{z^{k,i}_h-x^*_h}{g^{k,i}}
\overset{\text{\eqref{zki-xstar_proof}}} < \minAL \\
& \overset{\text{\eqref{initial_stepsize_armijo_2}}} \le \min\biggl\{A^{k,i},\dfrac{2(1-\gamma)}{L^{\text{max}}}\biggr\},
\end{split}
\end{equation}
implying that $\bar \alpha^{k,i} < A^{k,i}$.
Taking into account that the initial stepsize $\Delta^{k,i}$ in the Armijo line search is chosen as in~\eqref{initial_stepsize_armijo_1},
we have that $\Delta^{k,i} = \bar \alpha^{k,i}$. So, using again~\eqref{max_stepsize_proof} we obtain
\[
\Delta^{k,i} = \bar \alpha^{k,i} < \dfrac{2(1-\gamma)}{L^{\text{max}}} \le \dfrac{2(1-\gamma)}{L_{h,j(k)}}.
\]
From Lemma~\ref{lemma:lower_bound_stepsize} we get that $\alpha^{k,i} = \bar \alpha^{k,i}$.
Since $\bar \alpha^{k,i}$ is the largest feasible stepsize along $d^{k,i}$,
(at least) one variable between $z^{k,i+1}_h$ and $z^{k,i+1}_{j(k)}$ will be at the lower or the upper bound.
Using the fact that $k \ge \kz$, from Proposition~\ref{prop:z_j} we have that $z^{k,i+1}_{j(k)} \in (l_{j(k)},u_{j(k)})$, and then
$z^{k,i+1}_h$ will be necessarily at the lower or the upper bound.
Since $g^{k,i} < 0$, from the definition of the search direction given in~\eqref{g_d_def} it follows that $z^{k,i+1}_h = l_h$.
Using~\eqref{ac2cd_prop_1} and~\eqref{xstar_l}, we finally have that $z^{k,i+1}_h = x^{k+1}_h$ and $l_h = x^*_h$,
yielding to the desired result.
\end{proof}

\begin{remark}
From~\eqref{conv_z}, there must exist an outer iteration $\kA$ such that~\eqref{neighborhood_A} holds,
provided the whole sequence $\{x^k\}$ converges to $x^*$ and $\Astar^+(x^*) \ne \emptyset$.
\end{remark}

\section{Active-set complexity}\label{sec:active_set_complexity}
In this section, the main result of the paper is presented:
under convexity of $f$ and a quadratic growth condition (satisfied by any strongly convex function),
it is possible to compute the maximum number of iterations required by AC2CD to identify the active set,
thus extending what obtained in the previous section.
Using the definition given in~\cite{nutini:2019},
we refer to the maximum number of iterations required to identify the active set as ``active-set complexity''.

To obtain the desired result, we first show how choosing the initial stepsize in the Armijo line search, in order to meet an additional requirement.
Then, we will show non-asymptotic sublinear convergence rate of AC2CD,
which, combined with Theorem~\ref{th:active_set_general_1}, will lead to the active-set complexity of the algorithm.

\subsection{Initial stepsize in the Armijo line search}\label{subsec:stepsize}
To obtain non-asymptotic sublinear convergence rate of AC2CD, for all $k \ge 0$ we need to satisfy
\begin{equation}\label{req_A_D}
l_{j(k)} < z^{k,i}_{j(k)} < u_{j(k)}, \quad i = 1,\ldots,n+1.
\end{equation}

Note that, in general, \eqref{req_A_D} holds only for sufficiently large $k$
(see the proof of Theorem~1 in~\cite{cristofari:2019}).
To satisfy~\eqref{req_A_D} for all $k \ge 0$ we can
use sufficiently small stepsizes in all the inner iterations, exploiting the fact that
$x^k_{j(k)} = z^{k,1}_{j(k)} \in (l_{j(k)}, u_{j(k)})$ for all $k \ge 0$.
In particular, to obtain a small stepsize $\alpha^{k,i}$ from the Armijo line search we must choose a small value of the initial stepsize $\Delta^{k,i}$.
Taking into account~\eqref{initial_stepsize_armijo_1}, this means that we must use a small value of $A^{k,i}$.
Anyway, we have to keep in mind that $A^{k,i}$ must satisfy~\eqref{initial_stepsize_armijo_2} as well.
A possible strategy is setting $A_u>0$, $\epsilon \in (0,1)$ and, at every inner iteration $(k,i)$, computing
\begin{equation}\label{Aki_2}
A^{k,i} =
\begin{cases}
\min\{\hat \alpha^{k,i}, A_u\}, \quad & \text{if } g^{k,i} \ne 0 \text{ (i.e., if $d^{k,i}$ is a non-zero direction)}, \\
A_u, & \text{otherwise},
\end{cases}
\end{equation}
where $\hat \alpha^{k,i}$ is the stepsize such that $D_{j(k)}(z^{k,i} + \hat \alpha^{k,i}d^{k,i}) = \epsilon D_{j(k)}(z^{k,i})$ when $g^{k,i} \ne 0$.
Note that $\hat \alpha^{k,i}$ may be infeasible and/or infinity.
Since $\alpha^{k,i} \le A^{k,i} \le \hat \alpha^{k,i}$, it follows that
$D_{j(k)}(z^{k,i+1}) = D_{j(k)}(z^{k,i} + \alpha^{k,i} d^{k,i}) \ge D_{j(k)}(z^{k,i} + \hat \alpha^{k,i} d^{k,i}) \ge \epsilon D_{j(k)}(z^{k,i})$.
Consequently,
\begin{equation}\label{D_j_z}
D_{j(k)}(z^{k,i+1}) \ge \epsilon^i D_{j(k)}(z^{k,1}) = \epsilon^i D_{j(k)}(x^k) > 0, \quad i = 1,\ldots,n.
\end{equation}
Then, this choice of $A^{k,i}$ satisfies~\eqref{req_A_D} for all $k \ge 0$.
To show that it also satisfies~\eqref{initial_stepsize_armijo_2}, we have to explicitly write the expression of $\hat \alpha^{k,i}$,
which can be obtained by simple calculations
(recall that $\hat \alpha^{k,i}$ is defined only when $g^{k,i} \ne 0$):
%\allowdisplaybreaks
\begin{align*}
& \text{If $g^{k,i} > 0$,} \\
& \quad \hat \alpha^{k,i} =
\begin{cases}
(1-\epsilon)D_{j(k)}(z^{k,i})/g^{k,i}, & \text{if } D_{j(k)}(z^{k,i}) = z^{k,i}_{j(k)} - l_{j(k)}, \\
[z^{k,i}_{j(k)}-l_{j(k)}-\epsilon D_{j(k)}(z^{k,i})]/g^{k,i}, & \text{otherwise}; \\
\end{cases} \\
& \text{else if $g^{k,i} < 0$,} \\
& \quad \hat \alpha^{k,i} =
\begin{cases}
(1-\epsilon)D_{j(k)}(z^{k,i})/\abs{g^{k,i}}, & \text{if } D_{j(k)}(z^{k,i}) = u_{j(k)} - z^{k,i}_{j(k)}, \\
[u_{j(k)}-z^{k,i}_{j(k)}-\epsilon D_{j(k)}(z^{k,i})]/\abs{g^{k,i}}, & \text{otherwise}.
\end{cases}
\end{align*}
We see that, when $g^{k,i} \ne 0$,
we have $\hat \alpha^{k,i} \ge (1-\epsilon)D_{j(k)}(z^{k,i})/\abs{g^{k,i}}$.
Using~\eqref{Aki_2} and~\eqref{D_j_z}, it follows that
\[
\min\biggl\{\frac{(1-\epsilon)\epsilon^{i-1} D_{j(k)} (x^k)}{\abs{g^{k,i}}},A_u\biggr\} \le A^{k,i} \le A_u, \quad \text{if }  g^{k,i} \ne 0.
\]
Then,~\eqref{initial_stepsize_armijo_2} is satisfied with a proper value of $A_l$ which can be easily obtained,
since any non-zero $\abs{g^{k,i}}$ is less than or equal to $\max_{i,j = 1,\ldots,n} \{\nabla_j f(x)-\nabla_i f(x) \colon x \in \LS\}$
(which is finite by the assumption that the level set $\LS$ is compact) and, from~\eqref{j(k)}, we have
$\displaystyle{D_{j(k)}(x^k) \ge \tau \min_{x \in \LS} \, \max_{i = 1,\ldots,n} D_i(x)}$ (which is positive by Assumption~\ref{assumpt:l0_int_point}).

Many other strategies can be used to compute a value of $A^{k,i}$ that satisfies all the required conditions.
It is important to note that, in practice, $A^{k,i}$ should not be too small compared to the largest feasible stepsize $\bar \alpha^{k,i}$
(for a non-zero direction $d^{k,i}$),
otherwise the Armijo line search may produce extremely small stepsizes which can dramatically slow down the algorithm.
For example, $\epsilon$ should be sufficiently smaller than $1$ in the above described strategy.

In the rest of this section, we will assume $A^{k,i}$ to be computed in order to satisfy, together with~\eqref{initial_stepsize_armijo_2},
condition~\eqref{req_A_D} for all $k \ge 0$.

\subsection{Convergence rate analysis}
In this subsection we show that, when $f$ is convex, AC2CD has a non-asymptotic sublinear convergence rate.
Let us remark that the results reported here are completely different from those given in~\cite{cristofari:2019},
where a linear rate was obtained, but asymptotically,
whereas a non-asymptotic linear rate was proved only when there are no bounds on the variables
(both results are not useful in the analysis of the active-set complexity).

Our results here are obtained by adapting the analysis of the block coordinate gradient projection method in~\cite{beck:2013}
for minimization problems over the Cartesian product of closed convex sets.
In particular, with respect to~\cite{beck:2013}, major difficulties in our analysis come from the presence of the coupling constraint in the problem
and the absence of projection operations in the algorithm.
In such a context, the next lemma establishes a useful property of AC2CD.

\begin{lemma}\label{lemma:ac2cd_ineq}
For all $x^* \in X^*$, at every inner iteration $(k,i)$ of AC2CD we have that
\[
g^{k,i}(x^*_{p^k_i}-z^{k,i+1}_{p^k_i}) \le \T \abs{z^{k,i+1}_{p^k_i}-z^{k,i}_{p^k_i}}\abs{x^*_{p^k_i}-z^{k,i+1}_{p^k_i}},
\]
where
$A_l>0$ is the lower bound on the parameter $A^{k,i}$ used to compute the initial stepsize $\Delta^{k,i}$
in the Armijo line search (see~\eqref{initial_stepsize_armijo_1} and~\eqref{initial_stepsize_armijo_2}),
$\gamma \in (0,1)$ is the parameter for sufficient decrease in the Armijo line search,
$\delta \in (0,1)$ is the reduction parameter in the Armijo line search
and $L^{\text{max}}>0$ is the maximum among the local Lipschitz constants $L_{i,j}$, defined in~\eqref{Lmax}.
\end{lemma}

\begin{proof}
Consider any inner iteration $(k,i)$.
The result is trivial if $g^{k,i} = 0$, so we assume that $g^{k,i} \ne 0$ and distinguish two possible cases.
\begin{enumerate}[label=(\roman*), leftmargin=*]
\item First, assume that $\alpha^{k,i} = \bar \alpha^{k,i}$, that is, the largest feasible stepsize is used.
    This means that (at least) one variable between $z^{k,i+1}_{p^k_i}$ and $z^{k,i+1}_{j(k)}$ will be at the lower or the upper bound.
    Recalling that~\eqref{req_A_D} holds for all $k \ge 0$, necessarily $z^{k,i+1}_{p^k_i}$ will be at the lower or the upper bound.
    Using the definition of $\bar \alpha^{k,i}$ given in~\eqref{max_stepsize_def}, it follows that
    either $z^{k,i+1}_{p^k_i} = u_{p^k_i}$ if $g^{k,i} > 0$, or $z^{k,i+1}_{p^k_i} = l_{p^k_i}$ if $g^{k,i} < 0$,
    implying that $g^{k,i}(x^*_{p^k_i}-z^{k,i+1}_{p^k_i}) \le 0$ and the desired result is obtained.

\item Now, assume that $\alpha^{k,i} < \bar \alpha^{k,i}$, which implies that $\bar \alpha^{k,i} > 0$ and,
    from~\eqref{ac2cd_prop_3}, that $z^{k,i+1} \ne z^{k,i}$.
    Since $z^{k,i+1}_{p^k_i} = z^{k,i}_{p^k_i} + \alpha^{k,i} g^{k,i}$, it follows that $\alpha^{k,i} g^{k,i} \ne 0$.
    Recalling the definition of $g^{k,i}$ given in~\eqref{g_d_def}, this implies that $p^k_i \ne j(k)$.
    Moreover, we can write
    \[
    g^{k,i}(x^*_{p^k_i}-z^{k,i+1}_{p^k_i}) = \frac{(z^{k,i+1}_{p^k_i}-z^{k,i}_{p^k_i})(x^*_{p^k_i}-z^{k,i+1}_{p^k_i})}{\alpha^{k,i}}
    \le \frac{\abs{z^{k,i+1}_{p^k_i}-z^{k,i}_{p^k_i}}\abs{x^*_{p^k_i}-z^{k,i+1}_{p^k_i}}}{\alpha^{k,i}}.
    \]
    So, to obtain the desired result we have to show that
    \[
    \alpha^{k,i} \ge \min\biggl\{A_l,\frac{2\delta(1-\gamma)}{L^{\text{max}}}\biggr\}.
    \]
    To this extent, let us distinguish two further subcases, depending on
    whether $\Delta^{k,i} = \bar \alpha^{k,i}$ or $\Delta^{k,i} < \bar \alpha^{k,i}$,
    according to the definition of $\Delta^{k,i}$ given in~\eqref{initial_stepsize_armijo_1}.
    \begin{itemize}
    \item If $\Delta^{k,i} = \bar \alpha^{k,i}$, then $\alpha^{k,i} < \Delta^{k,i}$ (recall that we are considering the case $\alpha^{k,i} < \bar \alpha^{k,i}$)
        and, from Lemma~\ref{lemma:lower_bound_stepsize}, it follows that
        \[
        \alpha^{k,i} > \frac{2 \delta (1-\gamma)}{L_{p^k_i,j(k)}} \ge \min\biggl\{A_l, \frac{2 \delta (1-\gamma)}{L_{p^k_i,j(k)}}\biggr\}
        \ge \min\biggl\{A_l,\frac{2\delta(1-\gamma)}{L^{\text{max}}}\biggr\}.
        \]
    \item If $\Delta^{k,i} < \bar \alpha^{k,i}$, from~\eqref{initial_stepsize_armijo_1} we have $\Delta^{k,i} = A^{k,i}$.
        Using Lemma~\ref{lemma:lower_bound_stepsize} it follows that
        \begin{align*}
        \alpha^{k,i} & \ge \min\biggl\{\Delta^{k,i}, \frac{2 \delta (1-\gamma)}{L_{p^k_i,j(k)}}\biggr\}
                     = \min\biggl\{A^{k,i}, \frac{2 \delta (1-\gamma)}{L_{p^k_i,j(k)}}\biggr\} \\
                     & \ge \min\biggl\{A_l, \frac{2 \delta (1-\gamma)}{L_{p^k_i,j(k)}}\biggr\}
                     \ge \min\biggl\{A_l,\frac{2\delta(1-\gamma)}{L^{\text{max}}}\biggr\}.
        \end{align*}
    \end{itemize}
\end{enumerate}
\end{proof}

Now, we give a first result on the decrease in the objective function at every outer iteration.
\begin{proposition}\label{prop:obj_decrease_armijo}
At every outer iteration $k$ of AC2CD we have that
\[
f(x^k) - f(x^{k+1}) \ge \frac{\gamma}{A_u} \seminormjk{x^{k+1}-x^k}^2,
\]
where $A_u>0$ is the upper bound on the parameter $A^{k,i}$ used to compute the initial stepsize $\Delta^{k,i}$
in the Armijo line search (see~\eqref{initial_stepsize_armijo_1} and~\eqref{initial_stepsize_armijo_2}) and
$\gamma \in (0,1)$ is the parameter for sufficient decrease in the Armijo line search.
\end{proposition}

\begin{proof}
First we show that at, every inner iteration $(k,i)$, we have
\begin{equation}\label{obj_decrease_armijo}
f(z^{k,i}) - f(z^{k,i+1}) \ge \frac{\gamma}{A_u} (z^{k,i+1}_{p^k_i}-z^{k,i}_{p^k_i})^2.
\end{equation}
If $\alpha^{k,i} = 0$, then $z^{k,i+1} = z^{k,i}$ and~\eqref{obj_decrease_armijo} trivially holds.
If $\alpha^{k,i} > 0$, from the instructions of the Armijo line search it follows that
$f(z^{k,i+1}) \le f(z^{k,i}) + \gamma \alpha^{k,i} \nabla f(z^{k,i})^T d^{k,i}$.
Using~\eqref{ac2cd_prop_2}, we can write
\[
f(z^{k,i+1}) \le f(z^{k,i}) - \gamma \alpha^{k,i} (g^{k,i})^2 = f(z^{k,i}) - \frac{\gamma}{\alpha^{k,i}} (\alpha^{k,i} g^{k,i})^2.
\]
Since $z^{k,i+1}_{p^k_i} = z^{k,i} + \alpha^{k,i} g^{k,i}$ and $\alpha^{k,i} \le A_u$, we obtain~\eqref{obj_decrease_armijo}.
Hence, we have
\begin{align*}
f(x^k) - f(x^{k+1}) & = \sum_{i=1}^n [f(z^{k,i})-f(z^{k,i+1})] \overset{\text{\eqref{obj_decrease_armijo}}}
                      \ge \frac{\gamma}{A_u} \sum_{i=1}^n (z^{k,i+1}_{p^k_i}-z^{k,i}_{p^k_i})^2 \\
                    & = \frac{\gamma}{A_u} \sum_{i \colon p^k_i \ne j(k)} (z^{k,i+1}_{p^k_i}-z^{k,i}_{p^k_i})^2
                      \overset{\text{\eqref{ac2cd_prop_1}}} = \frac{\gamma}{A_u} \sum_{i \colon p^k_i \ne j(k)} (x^{k+1}_{p^k_i}-x^{k}_{p^k_i})^2 \\
                    & = \frac{\gamma}{A_u} \seminormjk{x^{k+1}-x^k}^2,
\end{align*}
where, in the second equality, we have used the fact that $z^{k,i+1}_{j(k)} = z^{k,i}_{j(k)}$
when $p^k_i = j(k)$, according to the definition of the search direction $d^{k,i}$ given in~\eqref{g_d_def}.
\end{proof}

In the rest of this section, the objective function will be required to be convex over $\R^n$
and its optimal value for problem~\eqref{prob} will be denoted by $f^*$.
Let us also define the following constants
(which are finite under convexity of $f$, since this implies $X^* \subseteq \LS$, where the level set $\LS$ is assumed to be non-empty and compact):
\begin{align}
& \Rzero = \max_{\substack{j=1,\ldots,n \\ x \in \LS \\ x^* \in X^*}} \seminormj{x-x^*}, \label{Rzero} \\
& \Gstar = \max_{\substack{i,j = 1,\ldots,n \\ x^* \in X^*}} [\nabla_j f(x^*)-\nabla_i f(x^*)]. \label{Gstar}
\end{align}
We see that $\Rzero$ is the maximum distance between a point in the level set $\LS$ and a point in $X^*$,
where the distance is measured in terms of the pseudometrics induced by the seminorms $\seminormj{\cdot}$
(the latter can be upper bounded by the Euclidean norm, see~\eqref{seminorm_le_norm}).
From the KKT conditions~\eqref{stat}, we also note that $\Gstar$ is related to the minimum strict complementarity measure $\zeta(x^*)$ defined in~\eqref{nondeg_measure},
in the sense that, if $\Astar^+(x^*) \ne \emptyset$ for some $x^* \in X^*$, then $\Gstar \ge \zeta(x^*) > 0$,
while, if $\Astar^+(x^*) = \emptyset$ for all $x^* \in X^*$, then \mbox{$\Gstar = 0$ and} $\zeta(x^*)$ is not defined for any $x^* \in X^*$.
We can interpret $\Gstar$ as a measure of the ``maximum amount of strict complementarity'' over \mbox{the set $X^*$}.

We now state a result which, for every outer iteration,
relates the decrease in the objective function with the optimization error.

\begin{proposition}\label{prop:obj_decrease}
Assume that $f$ is convex over $\R^n$. Then, at every outer iteration $k$ of AC2CD we have that
\[
f(x^k) - f(x^{k+1}) \ge \dfrac{\gamma (f(x^{k+1}) - f^*)^2}{A_u (n-1) \fdec^2},
\]
where $A_l>0$ and $A_u>0$ are the lower and the upper bound, respectively, on the parameter $A^{k,i}$
used to compute the initial stepsize $\Delta^{k,i}$ in the Armijo line search (see~\eqref{initial_stepsize_armijo_1} and~\eqref{initial_stepsize_armijo_2}),
$\delta \in (0,1)$ is the reduction parameter in the Armijo line search,
$\gamma \in (0,1)$ is the parameter for sufficient decrease in the Armijo line search,
$L^{\text{max}}>0$ is the maximum among the local Lipschitz constants $L_{i,j}$, defined as in~\eqref{Lmax},
$\Lmaxj>0$ is the maximum among the constants $L_j = \sum_{i=1}^n L_{i,j}$, defined as in~\eqref{Lmaxj},
$\Rzero \ge 0$ is the maximum distance between a point in the level set $\LS$ and an optimal solution, defined as in~\eqref{Rzero},
and $\Gstar \ge 0$ is the maximum strict complementarity measure over $X^*$, defined as in~\eqref{Gstar}.
\end{proposition}

\begin{proof}
Let $x^*$ be an optimal solution of problem~\eqref{prob} and consider any inner iteration $(k,i)$.
From the definition of the search direction $d^{k,i}$ given in~\eqref{g_d_def}, we have that $z^{k,i+1}_{p^k_i} \ge z^{k,i}_{p^k_i}$ if $g^{k,i} \ge 0$,
and $z^{k,i+1}_{p^k_i} \le z^{k,i}_{p^k_i}$ if $g^{k,i} \le 0$. Namely,
$g^{k,i} (z^{k,i}_{p^k_i}-z^{k,i+1}_{p^k_i}) \le 0$ and,
using~\eqref{ac2cd_prop_1}, we can write $g^{k,i} (x^k_{p^k_i}-x^{k+1}_{p^k_i}) \le 0$. Then,
\begin{equation*}
\begin{split}
g^{k,i}(x^k_{p^k_i}-x^*_{p^k_i}) & \le g^{k,i}(x^{k+1}_{p^k_i}-x^*_{p^k_i}) \\
                                 & = [\nabla_{p^k_i} f(x^{k+1}) - \nabla_{j(k)} f(x^{k+1})] (x^*_{p^k_i}-x^{k+1}_{p^k_i}) \\
                                 & \quad \, + [\nabla_{p^k_i} f(x^{k+1}) - \nabla_{j(k)} f(x^{k+1}) + g^{k,i}] (x^{k+1}_{p^k_i}-x^*_{p^k_i}). \\
\end{split}
\end{equation*}
Using Corollary~\ref{corol:lips} with $v=x^{k+1}$, we have that
\begin{equation*}
\begin{split}
\nabla_{p^k_i} f(x^{k+1}) - \nabla_{j(k)} f(x^{k+1}) + g^{k,i} & \overset{\phantom{\text{\eqref{ac2cd_prop_1}}}} \le \Lmaxj \seminormjk{z^{k,i}-x^{k+1}} \\
                                                               & \overset{\text{\eqref{ac2cd_prop_1}}} \le \Lmaxj \seminormjk{x^k-x^{k+1}}.
\end{split}
\end{equation*}
It follows that
\begin{equation*}
\begin{split}
g^{k,i}(x^k_{p^k_i}-x^*_{p^k_i}) & \le [\nabla_{p^k_i} f(x^{k+1}) - \nabla_{j(k)} f(x^{k+1})] (x^*_{p^k_i}-x^{k+1}_{p^k_i}) \\
                                 & \quad \, + \Lmaxj \seminormjk{x^k-x^{k+1}} \abs{x^*_{p^k_i}-x^{k+1}_{p^k_i}}.
\end{split}
\end{equation*}
Summing these inequalities, we obtain
\begin{equation*}
\begin{split}
\sum_{i \colon p^k_i \ne j(k)} g^{k,i}(x^k_{p^k_i}-x^*_{p^k_i})
& \overset{\phantom{\text{\eqref{sum_abs_vs_seminorm}}}} \le \sum_{i \colon p^k_i \ne j(k)}[\nabla_{p^k_i} f(x^{k+1}) - \nabla_{j(k)} f(x^{k+1})] (x^*_{p^k_i}-x^{k+1}_{p^k_i}) \\
& \quad \, + \Lmaxj \seminormjk{x^k-x^{k+1}} \sum_{i \colon p^k_i \ne j(k)}\abs{x^*_{p^k_i}-x^{k+1}_{p^k_i}} \\
& \overset{\text{\eqref{sum_abs_vs_seminorm}}} \le
  \sum_{i \colon p^k_i \ne j(k)} [\nabla_{p^k_i} f(x^{k+1}) - \nabla_{j(k)} f(x^{k+1})] (x^*_{p^k_i}-x^{k+1}_{p^k_i}) \\
& \quad \, + \sqrt{n-1} \, \Rzero \Lmaxj \seminormjk{x^k-x^{k+1}}.
\end{split}
\end{equation*}
Using Lemma~\ref{lemma:prod} with $v = \nabla f(x^{k+1})$, $x' = x^*$ and $x'' = x^{k+1}$, we can write
\begin{equation*}
\begin{split}
\nabla f(x^{k+1})^T (x^*-x^{k+1}) & = \innprod{\nabla f(x^{k+1})-\nabla_{j(k)} f(x^{k+1})e, x^*-x^{k+1}}_{j(k)} \\
                                  & = \sum_{i \colon p^k_i \ne j(k)} [\nabla_{p^k_i} f(x^{k+1}) - \nabla_{j(k)} f(x^{k+1})] (x^*_{p^k_i}-x^{k+1}_{p^k_i}),
\end{split}
\end{equation*}
and then,
\[
\sum_{i \colon p^k_i \ne j(k)} g^{k,i}(x^k_{p^k_i}-x^*_{p^k_i})
\le \nabla f(x^{k+1})^T (x^*-x^{k+1}) + \sqrt{n-1} \, \Rzero \Lmaxj \seminormjk{x^k-x^{k+1}}.
\]
From the convexity of $f$ we have that $f(x^{k+1})-f^* \le \nabla f(x^{k+1})^T (x^{k+1}-x^*)$. Hence,
\begin{equation*}
\begin{split}
f(x^{k+1})-f^* & \le \sum_{i \colon p^k_i \ne j(k)} g^{k,i}(x^*_{p^k_i}-x^k_{p^k_i}) + \sqrt{n-1} \, \Rzero \Lmaxj \seminormjk{x^k-x^{k+1}} \\
& = \sum_{i \colon p^k_i \ne j(k)} g^{k,i}(x^*_{p^k_i}-x^{k+1}_{p^k_i}) + \sum_{i \colon p^k_i \ne j(k)} g^{k,i}(x^{k+1}_{p^k_i}-x^k_{p^k_i}) \\
& \quad \, + \sqrt{n-1} \, \Rzero \Lmaxj \seminormjk{x^k-x^{k+1}}.
\end{split}
\end{equation*}
Using~\eqref{ac2cd_prop_1} and Lemma~\ref{lemma:ac2cd_ineq},
for all $i$ such that $p^k_i \ne j(k)$ we can write
\[
g^{k,i}(x^*_{p^k_i}-x^{k+1}_{p^k_i}) \le \T \abs{x^{k+1}_{p^k_i}-x^k_{p^k_i}}\abs{x^*_{p^k_i}-x^{k+1}_{p^k_i}}.
\]
Therefore,
\begin{equation}\label{f_decr_proof}
\begin{split}
f(x^{k+1})-f^* \le & \T \sum_{i \colon p^k_i \ne j(k)} \abs{x^{k+1}_{p^k_i}-x^k_{p^k_i}}\abs{x^*_{p^k_i}-x^{k+1}_{p^k_i}} \\
                   & + \sum_{i \colon p^k_i \ne j(k)} g^{k,i}(x^{k+1}_{p^k_i}-x^k_{p^k_i}) \\
                   & + \sqrt{n-1} \Rzero \Lmaxj \seminormjk{x^k-x^{k+1}}.
\end{split}
\end{equation}
To obtain the desired result, now we upper bound the two summations in the right-hand side of~\eqref{f_decr_proof} by appropriate constants.
\begin{itemize}
\item As for the first summation in the right-hand side of~\eqref{f_decr_proof}, using the fact that
    $\abs{x^*_{p^k_i}-x^{k+1}_{p^k_i}} \le \seminormjk{x^*-x^{k+1}}$ by~\eqref{abs_vs_seminorm},
    we can write
    \begin{equation}\label{f_decr_proof_ub1}
    \begin{split}
    \sum_{i \colon p^k_i \ne j(k)} \abs{x^{k+1}_{p^k_i}-x^k_{p^k_i}}\abs{x^*_{p^k_i}-x^{k+1}_{p^k_i}}
    & \overset{\phantom{\text{\eqref{sum_abs_vs_seminorm}}}} \le \Rzero \sum_{i \colon p^k_i \ne j(k)} \abs{x^{k+1}_{p^k_i}-x^k_{p^k_i}} \\
    & \overset{\text{\eqref{sum_abs_vs_seminorm}}} \le \sqrt{n-1}\Rzero\seminormjk{x^{k+1}-x^k}.
    \end{split}
    \end{equation}
\item As for the second summation in the right-hand side of~\eqref{f_decr_proof}, from the triangular inequality we have that
    \[
    g^{k,i} \le \abs{g^{k,i} + \nabla_{p^k_i} f(x^*) - \nabla_{j(k)} f(x^*)} + \abs{\nabla_{p^k_i} f(x^*) - \nabla_{j(k)} f(x^*)},
    \]
    and then, using Corollary~\ref{corol:lips} with $v = x^*$,
    \[
    g^{k,i} \le L_{j(k)} \seminormjk{x^*-z^{k,i}} + \abs{\nabla_{p^k_i} f(x^*) - \nabla_{j(k)} f(x^*)} \le \Lmaxj \Rzero + \Gstar.
    \]
    Taking into account~\eqref{sum_abs_vs_seminorm}, we get
    \begin{equation}\label{f_decr_proof_ub2}
    \sum_{i \colon p^k_i \ne j(k)} g^{k,i}(x^{k+1}_{p^k_i}-x^k_{p^k_i}) \le (\Lmaxj \Rzero + \Gstar) \sqrt{n-1} \seminormjk{x^{k+1}-x^k}.
    \end{equation}
\end{itemize}
Combining~\eqref{f_decr_proof} with~\eqref{f_decr_proof_ub1} and~\eqref{f_decr_proof_ub2}, we have that
\[
f(x^{k+1})-f^* \le \sqrt{n-1} \fdec \seminormjk{x^{k+1}-x^k}.
\]
Using Proposition~\ref{prop:obj_decrease_armijo}, the desired result is finally obtained.
\end{proof}

We are now ready to show the non-asymptotic sublinear convergence rate of AC2CD.
\begin{theorem}\label{th:rate}
Assume that $f$ is convex over $\R^n$. Then, at every outer iteration $k \ge 1$ of AC2CD we have that
\[
f(x^k) - f^* \le \frac C k,
\]
where $C$ is equal to
\[
\rate,
\]
$A_l>0$ and $A_u>0$ are the lower and the upper bound, respectively, on the parameter $A^{k,i}$
used to compute the initial stepsize $\Delta^{k,i}$ in the Armijo line search (see~\eqref{initial_stepsize_armijo_1} and~\eqref{initial_stepsize_armijo_2}),
$\delta \in (0,1)$ is the reduction parameter in the Armijo line search,
$\gamma \in (0,1)$ is the parameter for sufficient decrease in the Armijo line search,
$L^{\text{max}}>0$ is the maximum among the local Lipschitz constants $L_{i,j}$, defined as in~\eqref{Lmax},
$\Lmaxj>0$ is the maximum among the constants $L_j = \sum_{i=1}^n L_{i,j}$, defined as in~\eqref{Lmaxj},
$\Rzero \ge 0$ is the maximum distance between a point in the level set $\LS$ and an optimal solution, defined as in~\eqref{Rzero},
and $\Gstar \ge 0$ is the maximum strict complementarity measure over $X^*$, defined as in~\eqref{Gstar}.
\end{theorem}

\begin{proof}
Consider a sequence $\{a^k\}$ of nonnegative scalars such that $a^k - a^{k+1} \ge \beta (a^{k+1})^2$, for all $k \ge 0$, with $\beta >0$.
From Lemma~6.2 in~\cite{beck:2013} we have that, if $a^1 \le 3/(2\beta)$ and $a^2 \le 3/(4\beta)$, then $a^k \le 3/(2\beta k)$, for all $k \ge 1$.
Using $a^k = f(x^k) - f^*$,
in view of Proposition~\ref{prop:obj_decrease} we have that $a^k - a^{k+1} \ge \beta (a^{k+1})^2$ with $\beta \ge 3/(2C)$.
It follows that the desired result is obtained if
\begin{equation}\label{err_x1_and_x2}
f(x^1)-f^* \le C \quad \text{and} \quad f(x^2)-f^* \le \frac C 2.
\end{equation}
To show that~\eqref{err_x1_and_x2} holds, by definition of $C$ we first write
\begin{equation}\label{ineq_C_1}
\begin{split}
C & \ge \frac{\sqrt{n-1}}{L^{\text{max}}} \fdecden^2 \\
  & \ge \frac{\sqrt{n-1}}{L^{\text{max}}} \biggl[\biggl(\frac{L^{\text{max}}}2+ 2\Lmaxj\biggr)\Rzero + 2\Gstar\biggr]^2,
\end{split}
\end{equation}
where the last inequality follows from the fact that $2\delta(1-\gamma) \le 2$, since $\delta, \gamma \in (0,1)$.
Now, we use the trivial inequality $(\theta_1 + \theta_2 + \theta_3)^2 \ge 2 \theta_1 (\theta_2 + \theta_3)$,
holding for all $\theta_1,\theta_2,\theta_3 \in \R$, with the choice
$\theta_1 = L^{\text{max}}\Rzero/2, \theta_2 = 2\Lmaxj\Rzero, \theta_3 = 2\Gstar$. We get
\begin{equation}\label{ineq_C_2}
C \ge 2\sqrt{n-1} \Rzero (\Lmaxj\Rzero + \Gstar) \ge 2\biggl[\frac{\Lmaxj}2(\Rzero)^2 + \sqrt{n-1}\Gstar\Rzero\biggr],
\end{equation}
where the last inequality follows from the fact that we are assuming $n \ge 2$.

Now consider an outer iteration $k \ge 1$,
picking any $x^* \in X^*$ and any $j \in \{1,\ldots,n\}$.
Using Lemma~\ref{lemma:prod} with $v = \nabla f(x^*)$, $x'=x^k$ and $x''=x^*$, we have
\begin{equation*}
\begin{split}
\nabla f(x^*)^T (x^k-x^*) & \overset{\phantom{\text{\eqref{sum_abs_vs_seminorm}}}} = \innprod{\nabla f(x^*)-\nabla_j f(x^*)e, x^k-x^*}_j \\
                          & \overset{\phantom{\text{\eqref{sum_abs_vs_seminorm}}}} = \sum_{h \ne j} [\nabla_h f(x^*) - \nabla_j f(x^*)] (x^k_h-x^*_h) \\
                          & \overset{\text{\eqref{sum_abs_vs_seminorm}}} \le \sqrt{n-1} \Gstar \seminormj{x^k-x^*}.
\end{split}
\end{equation*}
So, using Lemma~\ref{lemma:lips_descent} with $x' = x^*$ and $x'' = x^k$, we get
\[
f(x^k) - f^* \le \nabla f(x^*)^T (x^k-x^*) + \frac {L_j}2 \seminormj{x^*-x^k}^2
             \le \sqrt{n-1} \Gstar \Rzero + \frac {\Lmaxj}2 (\Rzero)^2.
\]
In view of~\eqref{ineq_C_2}, we conclude that $f(x^k)-f^* \le C/2$, implying that~\eqref{err_x1_and_x2} holds.
\end{proof}

A question that can naturally arise is whether the constant $C$ in Theorem~\ref{th:rate} is tight.
To answer this challenging question, we can look in detail at the steps of the above proofs, from which it seems
that $C$ may in fact be loose.
For example, in the proof of Theorem~\ref{th:rate} we got a lower bound for $C$ by decomposing the last term in~\eqref{ineq_C_1}
as the sum of $\Lmaxj(\Rzero)^2 + 2\sqrt{n-1}\Gstar\Rzero$ and
\[
(2\sqrt{n-1}-1)\Lmaxj(\Rzero)^2 + \sqrt{n-1}\biggl[\frac{L^{\text{max}}(\Rzero)^2}4 + 4\frac{(\Lmaxj\Rzero)^2 + (\Gstar)^2 + 2\Lmaxj\Gstar\Rzero}{L^{\text{max}}}\biggr].
\]
We then obtained~\eqref{ineq_C_2} by lower bounding the above quantity by $0$.
But the above quantity may be much larger than $0$
and, for large values of $n$ and $\Gstar$, even dominant over $\Lmaxj(\Rzero)^2 + 2\sqrt{n-1}\Gstar\Rzero$,
observing that $\Lmaxj = \xi L^{\text{max}}$, with $\xi \in [1,n-1]$,
as we see from~\eqref{Lmax}, \eqref{Lj} and~\eqref{Lmaxj}.

In the literature, a non-asymptotic convergence rate was also shown for other coordinate descent methods on
different settings with one or more linear constraints,
where the working set is chosen by random selection~\cite{necoara:2013,necoara:2017,necoara:2014,patrascu:2015,she:2017}
or by rules based on first-order optimality violation~\cite{beck:2014,jaggi:2015}.
In particular, just like AC2CD, random coordinate descent do not use $\nabla f$ to choose the working set.
A sublinear rate (in expectation) with respect to the objective values was shown
for random coordinate descent in~\cite{necoara:2014} under convexity of $f$,
and a linear rate (in expectation) was shown in~\cite{she:2017} under the additional assumption of proximal-PL inequality.
We note that the sublinear rate $f(x^k)-f^* \le n^2L(\bar \Rzero)^2/[k + n^2L(\bar \Rzero)^2/(f(x^0)-f^*)]$
obtained for random coordinate descent in~\cite{necoara:2014},
where $\bar \Rzero = \max_x\{ \max_{x^* \in X^*} \norm{x-x^*} \colon x \in \LS\}$,
holds with respect to the inner iterations, so $k$ should be multiplied by a factor $\mathcal O(n)$
to have a fair comparison with AC2CD, for which the rate was computed with respect to the outer iterations.
With this adjustment, the rate of random coordinate descent is however better than $f(x^k)-f^* \le  C/k$ obtained for AC2CD,
with the constant $C$ from Theorem~\ref{th:rate} being
$\mathcal O(n(nL^{\text{max}}\Rzero+\Gstar)^2)$ if we reasonably assume $\sqrt n \gg 1/L^{\text{max}}$
and consider $\Lmaxj = \mathcal O (nL^{\text{max}})$ (since $\Lmaxj = \xi L^{\text{max}}$, with $\xi \in [1,n-1]$,
as observed above), where $L^{\text{max}} \le 2L$ from~\eqref{local_lipschitz}.

These results seem in agreement with the unconstrained case,
where cyclic coordinate selection achieves worse convergence rate than random selection and Gauss-Southwell-type rules~\cite{beck:2013,nutini:2015},
even if practical performances of the algorithms usually depend on the specific features of the problems.

\subsection{Computation of the active-set complexity}\label{subsec:active_set_complexity}
Using all the previous results, we can now compute the active-set complexity of AC2CD, that is,
the maximum number of iterations required by the algorithm to identify the active set.
In particular, we give an upper bound for $\kA$ appearing in Theorem~\ref{th:active_set_general_1}
under convexity of $f$ and a quadratic growth condition, which is now described.

We assume that there exists $\mu>0$ such that
\begin{equation}\label{quadratic}
f(x) - f^* \ge \frac{\mu} 2 \norm{x-x^*}^2, \quad \forall \, x \in \LS,
\end{equation}
where $x^* \in X^*$.
Note that~\eqref{quadratic} is automatically satisfied if $f$ is $\mu$-strongly convex over $\LS$~\cite{nesterov:2013}.
However,~\eqref{quadratic} is a weaker condition than strong convexity of $f$ over $\LS$, since
there exist convex functions that satisfy~\eqref{quadratic} even if they are non-strongly convex.
This can be seen in the following example, obtained from~\cite{necoara:2019} with proper adjustments.
Note that, in the provided example, $f$ is not even strictly convex, there is a unique optimal solution $x^*$ (so that $\{x^k\} \to x^*$)
and Assumption~\ref{assumpt:l0_int_point} is satisfied.

\begin{example}
Consider the following convex problem:
\begin{equation*}
\begin{split}
& \min \, f(x) = \frac12 x_1^2 + \sum_{i=2}^n x_i \\
& e^T x = 0 \\
& x_1 \ge -1 \\
& x_i \ge 0, \quad i = 2,\ldots,n,
\end{split}
\end{equation*}
with arbitrary dimension $n \ge 3$. Since the smallest eigenvalue of the Hessian matrix of $f$ is equal to $0$,
then $f$ is not strongly convex.
Actually, $f$ is not even strictly convex, since $f(\omega x' + (1-\omega)x'') = \omega f(x') + (1-\omega)f(x'')$ for all $\omega \in [0,1]$
and any distinct feasible points $x',x''$ such that $x'_1 = x''_1$.
We also have that $x^* = 0$ is the unique optimal solution and $f^* = 0$.
We conclude that~\eqref{quadratic} is satisfied with $\mu = 1$, since
$f(x) - f^* = \frac12 x_1^2 + \sum_{i=2}^n x_i \ge \frac12 \sum_{i=1}^n x^2_i = \frac12 ||x-x^*||^2$ for all feasible $x$.
\end{example}

\begin{theorem}\label{th:active_set_complexity_1}
The following upper bound holds for $\kA$ appearing in Theorem~\ref{th:active_set_general_1}
if $f$ is convex over $\R^n$ and statisfies~\eqref{quadratic}:
\[
\kA \le \left\lfloor\frac{2C}{\mu}\max\left\{\biggl(\neighj\biggr)^{-2}, \left(\neighA\right)^{-2}\right\}\right\rfloor+1,
\]
where $C \ge 0$ is the constant of the sublinear convergence rate defined in Theorem~\ref{th:rate},
$\Dmaxstar>0$ is the maximum distance from the bounds at $x^*$, defined as in Proposition~\ref{prop:j_active_set},
$\zeta(x^*)>0$ is the minimum strict complementarity measure at $x^*$, defined as in~\eqref{nondeg_measure},
$L$ is the Lipschitz constant of $\nabla f$,
$A_l>0$ is the lower bound on the parameter $A^{k,i}$ used to compute the initial stepsize $\Delta^{k,i}$
in the Armijo line search (see~\eqref{initial_stepsize_armijo_1} and~\eqref{initial_stepsize_armijo_2}),
$L^{\text{max}}>0$ is the maximum among the local Lipschitz constants $L_{i,j}$, defined as in~\eqref{Lmax},
and $\tau \in (0,1]$ is the parameter used to choose $j(k)$, satisfying~\eqref{j(k)}.
\end{theorem}

\begin{proof}
By the definition of $\kA$ given in Theorem~\ref{th:active_set_general_1}, it holds that $\kA \ge \max\{\kj,\kz\}$ and~\eqref{neighborhood_A} is satisfied.
Recalling the definition of $\kz$ given in Proposition~\ref{prop:z_j} and the fact that~\eqref{req_A_D} holds for all $k \ge 0$,
we have $\kz = 0$, and then $\kA \ge  \max\{\kj,\kz\} = \kj$.
So, from~\eqref{neighborhood_A} and the definition of $\kj$ given in Proposition~\ref{prop:j_active_set},
it follows that $\kA$ is the first outer iteration such that
\begin{subequations}\label{kA_proof}
\begin{align}
\norm{x^k-x^*}_{\infty} & < \neighj, \quad & \forall \, k \ge \kA. \label{kA_proof_1} \\
\norm{z^{k,i}-x^*} & < \neighA, \quad i = 1,\ldots,n, & \forall \, k \ge \kA. \label{kA_proof_2}
\end{align}
\end{subequations}
%Using the fact that $f$ is $\mu$-strongly convex over $\LS$, from known results~\cite{nesterov:2013} we get
%\[
%f(x^k) \ge f^* + \nabla f(x^*)^T (x^k-x^*) + \frac{\mu} 2 \norm{x^k-x^*}^2 \ge f^* + \frac{\mu} 2 \norm{x^k-x^*}^2, \quad \forall \, k \ge 0,
%\]
%where the last inequality follows from the fact that $x^*$ is an optimal solution.
By Theorem~\ref{th:rate} and~\eqref{quadratic}, for all $k \ge 1$ we hence have that
\begin{align*}
\norm{x^k-x^*}_{\infty}^2 \le \norm{x^k-x^*}^2 \le \frac2{\mu}[f(x^k) - f^*] & \le \frac{2C}{\mu k}, \\
\norm{z^{k,i}-x^*}^2 \le \frac2{\mu}[f(z^{k,i}) - f^*] \le \frac2{\mu}[f(x^k) - f^*] & \le \frac{2C}{\mu k}, \quad i = 1,\ldots,n,
\end{align*}
where, in the last chain of inequalities, we used the fact that $f(z^{k,i+1}) \le f(z^{k,i}) \le f(x^k)$, $i = 1,\ldots,n$.
Therefore, \eqref{kA_proof} holds for all $k$ such that
$\sqrt{2C/(\mu k)}$ is less than both the right-hand side of~\eqref{kA_proof_1}
and the right-hand side of~\eqref{kA_proof_2}, yielding to the upper bound for $\kA$ given in the assertion.
\end{proof}

We remark that Theorem~\ref{th:active_set_complexity_1} requires convexity and quadratic growth,
but it uses the convergence rate result stated in Theorem~\ref{th:rate}, holding for general convex objective functions.
As a consequence, we expect the upper bound provided for $\kA$ in Theorem~\ref{th:active_set_complexity_1} to be loose.
Improving the convergence rate of the algorithm under the additional quadratic growth condition
may hence be a challenging question, since it affects the active-set complexity.

\section{Additional results}\label{sec:final_remarks}
So far we have shown that AC2CD identifies $\Astar^+(x^*)$ in a finite number $\kA$ of outer iterations (provided $\{x^k\} \to x^*$),
also giving an upper bound for $\kA$ when $f$ is convex and satisfies a quadratic growth condition.

Now, we want to show that the counterparts of these results hold as well, in the sense that AC2CD is able to identify
the complement of $\Astar(x^*)$, the so called \textit{non-active set}, in a finite number $\kN$ of outer iterations,
where an upper bound for $\kN$ can be computed when $f$ is convex and satisfies~\eqref{quadratic}.
More specifically, still considering a sequence $\{x^k\} \to x^*$, we want to show that, for all $k > \kN$,
\begin{equation}\label{non_active_set_identification}
x^k_h \in (l_h,u_h), \quad \forall \, h \notin \Astar(x^*).
\end{equation}
Actually,~\eqref{non_active_set_identification} is quite obvious (it follows from the properties of the limit),
but obtaining an upper bound for $\kN$ can be of interest.
In particular, if~\eqref{active_set_identification} and~\eqref{non_active_set_identification} hold for $k > \kA$ and $k > \kN$, respectively,
for all $k > \max\{\kA,\kN\}$ we have that
\[
\Astar^+(x^*) \subseteq \bigl\{i \colon x^k_i \in \{l_i ,u_i\}\bigr\} \subseteq \Astar(x^*). \\
\]
As a consequence, if $x^*$ is non-degenerate, for all $k > \max\{\kA,\kN\}$ it holds
\begin{equation}\label{active_set_identification_nondeg}
x^k_h \in \{l_h ,u_h\} \, \Leftrightarrow \, h \in \Astar(x^*),
\end{equation}
that is, the active set is exactly identified after $\max\{\kA,\kN\}$ outer iterations.

First we show that \eqref{non_active_set_identification} holds
for all sufficiently $k$, without any assumption on the convexity of $f$, provided the whole sequence $\{x^k\}$ converges.
\begin{theorem}\label{th:active_set_general_2}
Let $\{x^k\}$ be a sequence of points produced by AC2CD and assume that $\lim_{k \to \infty} x^k = x^*$.
Define the minimum non-zero distance from the bounds at $x^*$ as
\[
\Dminstar = \min_{i \notin \Astar(x^*)} D_i(x^*),
\]
which is well defined and positive by Assumption~\ref{assumpt:l0_int_point}, and let $\kN$ be the first outer iteration such that
\[
\norm{x^k-x^*}_{\infty} < \Dminstar, \quad \forall \, k > \kN.
\]
Then, for all $k > \kN$ we have that
\[
x^k_h \in (l_h,u_h), \quad \forall \, h \notin \Astar(x^*).
\]
\end{theorem}

\begin{proof}
Consider an outer iteration $k > \kN$ and any index $h \notin \Astar(x^*)$.
We have $\abs{x^k_h - x^*_h} \le \norm{x^k-x^*}_{\infty} < \Dminstar \le D_h(x^*)$,
implying that
\begin{equation}\label{Dh_proof_1}
x^k_h - l_h > x^*_h - l_h - D_h(x^*) \quad \text{and} \quad u_h - x^k_h > u_h - x^*_h - D_h(x^*).
\end{equation}
Therefore, we can write
\begin{equation*}
\begin{split}
D_h(x^k) & = \min\{x^k_h-l_h,u_h-x^k_h\} \overset{\text{\eqref{Dh_proof_1}}} > \min\{x^*_h-l_h,u_h-x^*_h\} - D_h(x^*) \\
         & = D_h(x^*) - D_h(x^*) = 0,
\end{split}
\end{equation*}
that is, $x^k_h \in (l_h,u_h)$.
\end{proof}

We finally give an upper bound for $\kN$ under the same assumptions used in Theorem~\ref{th:active_set_complexity_1}.
As in the previous section, also here we assume the parameter $A^k_i$ in the Armijo line search to be computed
in order to satisfy, together with~\eqref{initial_stepsize_armijo_2}, condition~\eqref{req_A_D} for all $k \ge 0$,
as explained in Subsection~\ref{subsec:stepsize}.

\begin{theorem}\label{th:active_set_complexity_2}
The following upper bound holds for $\kN$ appearing in Theorem~\ref{th:active_set_general_2}
if $f$ is convex over $\R^n$ and statisfies~\eqref{quadratic}:
\[
\kN \le \biggl\lfloor\frac{2C}{\mu} \bigl(\Dminstar\bigr)^{-2}\biggr\rfloor + 1,
\]
where $C \ge 0$ is the constant of the sublinear convergence rate defined in Theorem~\ref{th:rate},
and $\Dminstar>0$ is the minimum non-zero distance from the bounds at $x^*$, defined as in Theorem~\ref{th:active_set_general_2}.
\end{theorem}

\begin{proof}
Reasoning as in the proof of Theorem~\ref{th:active_set_complexity_1},
the desired result follows from Theorem~\ref{th:active_set_general_2} and the fact that
$\norm{x^k-x^*}^2_{\infty} \le \norm{x^k-x^*}^2 \le \frac{2C}{\mu k}$ for all $k \ge 1$.
\end{proof}

The same remarks stated after Theorem~\ref{th:active_set_complexity_1} hold for Theorem~\ref{th:active_set_complexity_2} as well.
Namely, we expect the upper bound provided for $\kN$ to be loose, since it requires convexity and quadratic growth,
but it uses the convergence rate result of Theorem~\ref{th:rate}, holding for general convex objective functions.

%\clearpage
\begin{appendices}
\section{Proofs of the technical results of Section~\ref{technical_results}}\label{app:technical results}
\begin{proof}[Proof of Lemma~\ref{lemma:prod}]
For all $x \in \F$ we have $x_j = b - \sum_{i \ne j} x_i$, $j=1,\ldots,n$.
So,
\begin{align*}
v^T (x'-x'') & = \sum_{i \ne j} v_i (x'_i-x''_i) + v_j (x'_j-x''_j) \\
             & = \sum_{i \ne j} v_i (x'_i-x''_i) - v_j \Biggl(\sum_{i \ne j} x_i' - \sum_{i \ne j} x_i''\Biggr)
             = \sum_{i \ne j} (v_i-v_j) (x'_i-x''_j).
\end{align*}
\end{proof}

\begin{proof}[Proof of Lemma~\ref{lemma:lips_const}]
Fix $j \in \{1,\ldots,n\}$ and $x',x'' \in \F$.
For all $i = 1,\dots,n$ and $x \in \F$, let $\phi_{i,j,x}$ be the functions appearing in~\eqref{local_lips_phi}.
Pick any $h \ne j$ and, from known results on functions with Lipschitz continuous derivatives~\cite{nesterov:2013}, we can write
\begin{align*}
f(x + t (e_h-e_j)) & = \phi_{h,j,x} (t) \le \phi_{h,j,x} (0) + t \dot \phi_{h,j,x} (0) + \frac{L_{h,j}} 2 t^2 \\
                   & = f(x) + t \nabla f(x)^T (e_h-e_j) + \frac{L_{h,j}} 2 t^2, \quad \forall \, t \in \R.
\end{align*}
Using $t = \dfrac 1{L_{h,j}} (\nabla_j f(x) - \nabla_h f(x))$, we get
\begin{equation}\label{f_diff}
f(x) - f\Bigl(x + \frac 1{L_{h,j}} (\nabla_j f(x) - \nabla_h f(x)) (e_h-e_j)\Bigr) \ge \frac 1{2L_{h,j}} (\nabla_h f(x) - \nabla_j f(x))^2.
\end{equation}
Let $\bar f = \inf_{x \in \R^n} f(x)$. For all $x \in \R^n$ we can write
\begin{equation}\label{f-fmin}
\begin{split}
f(x) - \bar f & \overset{\phantom{\text{($\ast$)}}} \ge f(x) - f\Bigl(x + \frac 1{L_{h,j}} (\nabla_j f(x) - \nabla_h f(x)) (e_h-e_j)\Bigr) \\
              & \overset{\phantom{\text{($\ast$)}}} \ge \frac 1 2 \max_{i \ne j} \frac 1{L_{i,j}} (\nabla_i f(x) - \nabla_j f(x))^2 \\
              & \overset{\text{($\ast$)}} \ge \frac 1 {2 \sum_{i \ne j} L_{i,j}} \sum_{i=1}^n (\nabla_i f(x) - \nabla_j f(x))^2 \\
              & \overset{\phantom{\text{($\ast$)}}} = \frac 1 {2 L_j} \sum_{i \ne j} (\nabla_i f(x) - \nabla_j f(x))^2
                = \frac 1 {2 L_j} \bigl\|\nabla f(x)-\nabla_j f(x) e\bigr\|_{\innprod j}^2,
\end{split}
\end{equation}
where the second inequality follows~\eqref{f_diff}, whereas the inequality $\text{($\ast$)}$ follows from the fact that
\[
\max_{i=1,\ldots,r} \frac{a_i}{b_i} \ge \frac 1{b_1+\ldots+b_r}\sum_{i=1}^n a_i,
\]
for all $a_1,\ldots,a_r \in \R$ and $b_1,\ldots,b_r > 0$.

Now, define the convex function $\psi_1(x) = f(x) - f(x') - \nabla f(x')^T (x-x')$.
Since $\nabla \psi_1(x) = \nabla f(x) - \nabla f(x')$, for all $x \in \F$, $i \in \{1,\ldots,n\}$ and $t,s \in \R$, we can write
\begin{align*}
  & \abs{\nabla \psi_1(x + t (e_i-e_j))^T (e_i-e_j) - \nabla \psi_1(x + s (e_i-e_j))^T (e_i-e_j)} \\
= & \abs{\nabla f(x + t (e_i-e_j))^T (e_i-e_j) - \nabla f(x + s (e_i-e_j))^T (e_i-e_j)} \le L_{i,j} \abs{t-s},
\end{align*}
where the last inequality follows from the fact that $L_{i,j}$ are local Lipschitz constants for $\nabla f(x)$.
Therefore, $L_{i,j}$ are also local Lipschitz constants for $\nabla \psi_1$.
Consequently, we can use~\eqref{f-fmin} with $f$ replaced by $\psi_1$. Observing that $\min_{x \in \R^n} \psi_1(x) = 0$, we obtain
\begin{align*}
\psi_1(x) & \ge \frac 1 {2 L_j} \bigl\|\nabla \psi_1(x)-\nabla_j \psi_1(x) e\bigr\|_{\innprod j}^2 \\
          & = \frac 1 {2 L_j} \bigl\|(\nabla f(x)-\nabla_j f(x)e)-(\nabla f(x')-\nabla_j f(x')e)\bigr\|_{\innprod j}^2, \quad \forall \, x \in \R^n.
\end{align*}
Using $x = x''$ in the above relation, we get
\begin{equation}\label{psi1}
\psi_1(x'') \ge \frac 1 {2 L_j} \bigl\|(\nabla f(x'')-\nabla_j f(x'')e)-(\nabla f(x')-\nabla_j f(x')e)\bigr\|_{\innprod j}^2.
\end{equation}
Defining the function $\psi_2(x) = f(x) - f(x'') - \nabla f(x'')^T (x-x'')$, we can reason as above and we obtain
\begin{equation}\label{psi2}
\psi_2(x') \ge \frac 1 {2 L_j} \bigl\|(\nabla f(x')-\nabla_j f(x')e)-(\nabla f(x'')-\nabla_j f(x'')e)\bigr\|_{\innprod j}^2.
\end{equation}
Summing~\eqref{psi1} and~\eqref{psi2}, we get
\[
\bigl\|[\nabla f(x')-\nabla_j f(x')e]-[\nabla f(x'')-\nabla_j f(x'')e]\bigr\|_{\innprod j}^2 \le L_j [\nabla f(x') - \nabla f(x'')]^T (x'-x'').
\]
So, to obtain the desired result we have to show that
$[\nabla f(x') - \nabla f(x'')]^T (x'-x'')$ is less than or equal to
\begin{equation}\label{prod_seminorm}
\bigl\|[\nabla f(x') - \nabla_j f(x')e]-[\nabla f(x'') - \nabla_j f(x'')e]\bigr\|_{\innprod j} \, \seminormj{x'-x''}.
\end{equation}
This can be achieved by using Lemma~\ref{lemma:prod} first with $v = \nabla f(x')$ and then with $v = \nabla f(x'')$, in order
to rewrite $[\nabla f(x') - \nabla f(x'')]^T (x'-x'')$ as
\[
\innprod{[\nabla f(x') - \nabla_j f(x')e]-[\nabla f(x'') - \nabla_j f(x'')e],x'-x''}_j.
\]
Hence, by using inequality~\eqref{cauchy} we obtain that the above quantity is less than or equal to~\eqref{prod_seminorm}.
\end{proof}

\begin{proof}[Proof of Corollary~\ref{corol:lips}]
From~\eqref{abs_vs_seminorm} and the definition of $g^{k,i}$ given in~\eqref{g_d_def}, for all $v \in \R^n$ we have that
\[
\abs{\nabla_{p^k_i} f(v)-\nabla_{j(k)} f(v)+g^{k,i}} \le \bigl\|[\nabla f(v) - \nabla_{j(k)} f(v) e]-[\nabla f(z^{k,i}) - \nabla_{j(k)} f(z^{k,i}) e]\bigr\|_{\innprod{j(k)}}.
\]
Using Lemma~\ref{lemma:lips_const}, the desired result is obtained.
\end{proof}

\begin{proof}[Proof of Lemma~\ref{lemma:lips_descent}]
Fix $j \in \{1,\ldots,n\}$ and $x',x'' \in \F$.
From the mean value theorem and using Lemma~\ref{lemma:prod} with $v = \nabla f(x' + t(x''-x'))$, we have
\begin{equation*}
\begin{split}
f(x'') - f(x') & = \int_0^1 \nabla f(x' + t(x''-x'))^T (x''-x') \, dt \\
       & = \int_0^1 \innprod{\nabla f(x' + t(x''-x'))-\nabla_j f(x' + t(x''-x'))e, x''-x'}_j \, dt.
\end{split}
\end{equation*}
The integrand in the last term of the above chain of equalities can be rewritten as the sum of $\innprod{\nabla f(x')-\nabla_j f(x')e,x''-x'}_j$ and
\[
\innprod{[\nabla f(x' + t(x''-x'))-\nabla_j f(x' + t(x''-x'))e]-[\nabla f(x')-\nabla_j f(x')e], x''-x'}_j,
\]
and the latter, by using inequality~\eqref{cauchy} and Lemma~\ref{lemma:lips_const}, is less than or equal to $t L_j \seminormj{x'-x''}^2$.
Therefore,
\begin{equation*}
\begin{split}
f(x'') & \le f(x') + \innprod{\nabla f(x')-\nabla_j f(x')e,x''-x'}_j + L_j \seminormj{x'-x''}^2 \int_0^1 t \, dt \\
       & = f(x') + \innprod{\nabla f(x')-\nabla_j f(x')e,x''-x'}_j + \frac{L_j}2 \seminormj{x'-x''}^2.
\end{split}
\end{equation*}
Using Lemma~\ref{lemma:prod} with $v = \nabla f(x')$, the desired result is obtained.
\end{proof}
\end{appendices}

\section*{Acknowledgments}
The author would like to thank the two anonymous reviewers for their comments and suggestions.

\bibliography{cristofari_R2}

\begin{thebibliography}{51}
\providecommand{\natexlab}[1]{#1}
\providecommand{\url}[1]{\texttt{#1}}
\expandafter\ifx\csname urlstyle\endcsname\relax
  \providecommand{\doi}[1]{doi: #1}\else
  \providecommand{\doi}{doi: \begingroup \urlstyle{rm}\Url}\fi

\bibitem[Beck(2014)]{beck:2014}
A.~Beck.
\newblock {The 2-coordinate descent method for solving double-sided simplex
  constrained minimization problems}.
\newblock \emph{J. Optim. Theory Appl.}, 162\penalty0 (3):\penalty0 892--919,
  2014.

\bibitem[Beck and Tetruashvili(2013)]{beck:2013}
A.~Beck and L.~Tetruashvili.
\newblock {On the convergence of block coordinate descent type methods}.
\newblock \emph{SIAM J. Optim.}, 23\penalty0 (4):\penalty0 2037--2060, 2013.

\bibitem[Bertsekas(1976)]{bertsekas:1976}
D.~P. Bertsekas.
\newblock {On the Goldstein-Levitin-Polyak gradient projection method}.
\newblock \emph{IEEE Trans. Automat. Control}, 21\penalty0 (2):\penalty0
  174--184, 1976.

\bibitem[Birgin and Mart{\'\i}nez(2002)]{birgin:2002}
E.~G. Birgin and J.~M. Mart{\'\i}nez.
\newblock {Large-scale active-set box-constrained optimization method with
  spectral projected gradients}.
\newblock \emph{Comput. Optim. Appl.}, 23\penalty0 (1):\penalty0 101--125,
  2002.

\bibitem[Bomze et~al.(2019)Bomze, Rinaldi, and Bul{\`o}]{bomze:2019a}
I.~M. Bomze, F.~Rinaldi, and S.~R. Bul{\`o}.
\newblock {First-order Methods for the Impatient: Support Identification in
  Finite Time with Convergent Frank--Wolfe Variants}.
\newblock \emph{SIAM J. Optim.}, 29\penalty0 (3):\penalty0 2211--2226, 2019.

\bibitem[Bomze et~al.(2020)Bomze, Rinaldi, and Zeffiro]{bomze:2019b}
I.~M. Bomze, F.~Rinaldi, and D.~Zeffiro.
\newblock {Active Set Complexity of the Away-Step Frank-Wolfe Algorithm}.
\newblock \emph{SIAM J. Optim.}, 30\penalty0 (3):\penalty0 2470--2500, 2020.

\bibitem[Burke(1990)]{burke:1990}
J.~Burke.
\newblock {On the identification of active constraints II: The nonconvex case}.
\newblock \emph{SIAM J. Numer. Anal.}, 27\penalty0 (4):\penalty0 1081--1102,
  1990.

\bibitem[Burke and Mor{\'e}(1988)]{burke:1988}
J.~V. Burke and J.~J. Mor{\'e}.
\newblock {On the identification of active constraints}.
\newblock \emph{SIAM J. Numer. Anal.}, 25\penalty0 (5):\penalty0 1197--1211,
  1988.

\bibitem[Burke and Mor{\'e}(1994)]{burke:1994}
J.~V. Burke and J.~J. Mor{\'e}.
\newblock {Exposing constraints}.
\newblock \emph{SIAM J. Optim.}, 4\penalty0 (3):\penalty0 573--595, 1994.

\bibitem[Calamai and Mor{\'e}(1987)]{calamai:1987}
P.~H. Calamai and J.~J. Mor{\'e}.
\newblock {Projected gradient methods for linearly constrained problems}.
\newblock \emph{Math. Program.}, 39\penalty0 (1):\penalty0 93--116, 1987.

\bibitem[Clarkson(2010)]{clarkson:2010}
K.~L. Clarkson.
\newblock {Coresets, sparse greedy approximation, and the Frank-Wolfe
  algorithm}.
\newblock \emph{ACM Trans. Algorithms}, 6\penalty0 (4):\penalty0 1--30, 2010.

\bibitem[Cristofari(2019)]{cristofari:2019}
A.~Cristofari.
\newblock {An almost cyclic 2-coordinate descent method for singly linearly
  constrained problems}.
\newblock \emph{Comput. Optim. Appl.}, 73\penalty0 (2):\penalty0 411--452,
  2019.

\bibitem[Cristofari et~al.(2017)Cristofari, De~Santis, Lucidi, and
  Rinaldi]{cristofari:2017}
A.~Cristofari, M.~De~Santis, S.~Lucidi, and F.~Rinaldi.
\newblock {A Two-Stage Active-Set Algorithm for Bound-Constrained
  Optimization}.
\newblock \emph{J. Optim. Theory Appl.}, 172\penalty0 (2):\penalty0 369--401,
  2017.

\bibitem[Cristofari et~al.(2020{\natexlab{a}})Cristofari, De~Santis, Lucidi,
  and Rinaldi]{cristofari:2017simplex}
A.~Cristofari, M.~De~Santis, S.~Lucidi, and F.~Rinaldi.
\newblock {An active-set algorithmic framework for non-convex optimization
  problems over the simplex}.
\newblock \emph{Comput. Optim. Appl.}, 77:\penalty0 57--89, 2020{\natexlab{a}}.

\bibitem[Cristofari et~al.(2020{\natexlab{b}})Cristofari, Rinaldi, and
  Tudisco]{cristofari:2019total}
A.~Cristofari, F.~Rinaldi, and F.~Tudisco.
\newblock {Total variation based community detection using a nonlinear
  optimization approach}.
\newblock \emph{SIAM J. Appl. Math.}, 80\penalty0 (3):\penalty0 1392--1419,
  2020{\natexlab{b}}.

\bibitem[Daniilidis et~al.(2009)Daniilidis, Sagastiz{\'a}bal, and
  Solodov]{daniilidis:2009}
A.~Daniilidis, C.~Sagastiz{\'a}bal, and M.~Solodov.
\newblock {Identifying structure of nonsmooth convex functions by the bundle
  technique}.
\newblock \emph{SIAM J. Optim.}, 20\penalty0 (2):\penalty0 820--840, 2009.

\bibitem[De~Santis et~al.(2016)De~Santis, Lucidi, and Rinaldi]{desantis:2016}
M.~De~Santis, S.~Lucidi, and F.~Rinaldi.
\newblock {A Fast Active Set Block Coordinate Descent Algorithm for
  $\ell_1$-Regularized Least Squares}.
\newblock \emph{SIAM J. Optim.}, 26\penalty0 (1):\penalty0 781--809, 2016.

\bibitem[Duchi et~al.(2021)Duchi, Ruan, et~al.]{duchi:2016}
J.~C. Duchi, F.~Ruan, et~al.
\newblock {Asymptotic optimality in stochastic optimization}.
\newblock \emph{Ann. Statist.}, 49\penalty0 (1):\penalty0 21--48, 2021.

\bibitem[Dunn(1987)]{dunn:1987}
J.~C. Dunn.
\newblock {On the convergence of projected gradient processes to singular
  critical points}.
\newblock \emph{J. Optim. Theory Appl.}, 55\penalty0 (2):\penalty0 203--216,
  1987.

\bibitem[Facchinei and Pang(2003)]{facchinei:2003}
F.~Facchinei and J.-S. Pang.
\newblock \emph{{Finite-dimensional variational inequalities and
  complementarity problems}}.
\newblock Springer Science \& Business Media, 2003.

\bibitem[Facchinei et~al.(1998)Facchinei, J{\'u}dice, and
  Soares]{facchinei:1998}
F.~Facchinei, J.~J{\'u}dice, and J.~Soares.
\newblock {An active set Newton algorithm for large-scale nonlinear programs
  with box constraints}.
\newblock \emph{SIAM J. Optim.}, 8\penalty0 (1):\penalty0 158--186, 1998.

\bibitem[Gafni and Bertsekas(1984)]{gafni:1984}
E.~M. Gafni and D.~P. Bertsekas.
\newblock {Two-metric projection methods for constrained optimization}.
\newblock \emph{SIAM J. Control Optim.}, 22\penalty0 (6):\penalty0 936--964,
  1984.

\bibitem[Hager and Zhang(2006)]{hager:2006}
W.~W. Hager and H.~Zhang.
\newblock {A new active set algorithm for box constrained optimization}.
\newblock \emph{SIAM J. Optim.}, 17\penalty0 (2):\penalty0 526--557, 2006.

\bibitem[Hare(2011)]{hare:2011}
W.~Hare.
\newblock {Identifying active manifolds in regularization problems}.
\newblock In \emph{Fixed-Point Algorithms for Inverse Problems in Science and
  Engineering}, pages 261--271. Springer, 2011.

\bibitem[Hare(2009)]{hare:2009}
W.~L. Hare.
\newblock {A proximal method for identifying active manifolds}.
\newblock \emph{Comput. Optim. Appl.}, 43\penalty0 (2):\penalty0 295--306,
  2009.

\bibitem[Hare and Lewis(2004)]{hare:2004}
W.~L. Hare and A.~S. Lewis.
\newblock {Identifying active constraints via partial smoothness and
  prox-regularity}.
\newblock \emph{J. Convex Anal.}, 11\penalty0 (2):\penalty0 251--266, 2004.

\bibitem[Hsieh et~al.(2008)Hsieh, Chang, Lin, Keerthi, and
  Sundararajan]{hsieh:2008}
C.-J. Hsieh, K.-W. Chang, C.-J. Lin, S.~S. Keerthi, and S.~Sundararajan.
\newblock {A dual coordinate descent method for large-scale linear SVM}.
\newblock In \emph{Proceedings of the 25th international conference on Machine
  learning}, pages 408--415, 2008.

\bibitem[Jaggi and Lacoste-Julien(2015)]{jaggi:2015}
M.~Jaggi and S.~Lacoste-Julien.
\newblock {On the global linear convergence of frank-wolfe optimization
  variants}.
\newblock \emph{Advances in Neural Information Processing Systems}, 28, 2015.

\bibitem[Lee and Wright(2012)]{lee:2012}
S.~Lee and S.~J. Wright.
\newblock {Manifold identification in dual averaging for regularized stochastic
  online learning}.
\newblock \emph{J. Mach. Learn. Res.}, 13\penalty0 (Jun):\penalty0 1705--1744,
  2012.

\bibitem[Lewis and Wright(2011)]{lewis:2011}
A.~S. Lewis and S.~J. Wright.
\newblock {Identifying activity}.
\newblock \emph{SIAM J. Optim.}, 21\penalty0 (2):\penalty0 597--614, 2011.

\bibitem[Lewis and Torczon(2010)]{lewis:2010}
R.~M. Lewis and V.~Torczon.
\newblock {Active set identification for linearly constrained minimization
  without explicit derivatives}.
\newblock \emph{SIAM J. Optim.}, 20\penalty0 (3):\penalty0 1378--1405, 2010.

\bibitem[Liang et~al.(2017)Liang, Fadili, and Peyr{\'e}]{liang:2017}
J.~Liang, J.~Fadili, and G.~Peyr{\'e}.
\newblock {Activity Identification and Local Linear Convergence of
  Forward--Backward-type Methods}.
\newblock \emph{SIAM J. Optim.}, 27\penalty0 (1):\penalty0 408--437, 2017.

\bibitem[Lin et~al.(2009)Lin, Lucidi, Palagi, Risi, and Sciandrone]{lin:2009}
C.-J. Lin, S.~Lucidi, L.~Palagi, A.~Risi, and M.~Sciandrone.
\newblock {Decomposition algorithm model for singly linearly-constrained
  problems subject to lower and upper bounds}.
\newblock \emph{J. Optim. Theory Appl.}, 141\penalty0 (1):\penalty0 107--126,
  2009.

\bibitem[Luo and Tseng(1992)]{luo:1992}
Z.-Q. Luo and P.~Tseng.
\newblock {On the convergence of the coordinate descent method for convex
  differentiable minimization}.
\newblock \emph{J. Optim. Theory Appl.}, 72\penalty0 (1):\penalty0 7--35, 1992.

\bibitem[Luo and Tseng(1993)]{luo:1993}
Z.-Q. Luo and P.~Tseng.
\newblock {On the convergence rate of dual ascent methods for linearly
  constrained convex minimization}.
\newblock \emph{Math. Oper. Res.}, 18\penalty0 (4):\penalty0 846--867, 1993.

\bibitem[Mifflin and Sagastiz{\'a}bal(2002)]{mifflin:2002}
R.~Mifflin and C.~Sagastiz{\'a}bal.
\newblock {Proximal points are on the fast track}.
\newblock \emph{J. Convex Anal.}, 9\penalty0 (2):\penalty0 563--580, 2002.

\bibitem[Necoara(2013)]{necoara:2013}
I.~Necoara.
\newblock {Random coordinate descent algorithms for multi-agent convex
  optimization over networks}.
\newblock \emph{IEEE Trans. Automat. Control}, 58\penalty0 (8):\penalty0
  2001--2012, 2013.

\bibitem[Necoara and Patrascu(2014)]{necoara:2014}
I.~Necoara and A.~Patrascu.
\newblock {A random coordinate descent algorithm for optimization problems with
  composite objective function and linear coupled constraints}.
\newblock \emph{Computational Optimization and Applications}, 57\penalty0
  (2):\penalty0 307--337, 2014.

\bibitem[Necoara et~al.(2017)Necoara, Nesterov, and Glineur]{necoara:2017}
I.~Necoara, Y.~Nesterov, and F.~Glineur.
\newblock {Random block coordinate descent methods for linearly constrained
  optimization over networks}.
\newblock \emph{Journal of Optimization Theory and Applications}, 173\penalty0
  (1):\penalty0 227--254, 2017.

\bibitem[Necoara et~al.(2019)Necoara, Nesterov, and Glineur]{necoara:2019}
I.~Necoara, Y.~Nesterov, and F.~Glineur.
\newblock {Linear convergence of first order methods for non-strongly convex
  optimization}.
\newblock \emph{Math. Program.}, 175\penalty0 (1):\penalty0 69--107, 2019.

\bibitem[Nesterov(2013)]{nesterov:2013}
Y.~Nesterov.
\newblock \emph{{Introductory lectures on convex optimization: A basic
  course}}, volume~87.
\newblock Springer Science \& Business Media, 2013.

\bibitem[Nutini et~al.(2015)Nutini, Schmidt, Laradji, Friedlander, and
  Koepke]{nutini:2015}
J.~Nutini, M.~Schmidt, I.~Laradji, M.~Friedlander, and H.~Koepke.
\newblock {Coordinate descent converges faster with the gauss-southwell rule
  than random selection}.
\newblock In \emph{International Conference on Machine Learning}, pages
  1632--1641. PMLR, 2015.

\bibitem[Nutini et~al.(2017)Nutini, Laradji, and Schmidt]{nutini:2017}
J.~Nutini, I.~Laradji, and M.~Schmidt.
\newblock {Let's Make Block Coordinate Descent Go Fast: Faster Greedy Rules,
  Message-Passing, Active-Set Complexity, and Superlinear Convergence}.
\newblock \emph{preprint}, https://arxiv.org/abs/1712.08859, 2017.

\bibitem[Nutini et~al.(2019)Nutini, Schmidt, and Hare]{nutini:2019}
J.~Nutini, M.~Schmidt, and W.~Hare.
\newblock {``Active-set complexity'' of proximal gradient: How long does it
  take to find the sparsity pattern?}
\newblock \emph{Optim. Lett.}, 13\penalty0 (4):\penalty0 645--655, 2019.

\bibitem[Ortega and Rheinboldt(1970)]{ortega:1970}
J.~M. Ortega and W.~C. Rheinboldt.
\newblock \emph{{Iterative solution of nonlinear equations in several
  variables}}, volume~30.
\newblock Siam, 1970.

\bibitem[Patrascu and Necoara(2015)]{patrascu:2015}
A.~Patrascu and I.~Necoara.
\newblock {Efficient random coordinate descent algorithms for large-scale
  structured nonconvex optimization}.
\newblock \emph{J. Global Optim.}, 61\penalty0 (1):\penalty0 19--46, 2015.

\bibitem[Poon et~al.(2018)Poon, Liang, and Schoenlieb]{poon:2018}
C.~Poon, J.~Liang, and C.~Schoenlieb.
\newblock {Local convergence properties of SAGA/Prox-SVRG and acceleration}.
\newblock In \emph{Proc. Mach. Learn. Res. (PMLR)}, pages 4124--4132, 2018.

\bibitem[She and Schmidt(2017)]{she:2017}
J.~She and M.~Schmidt.
\newblock {Linear convergence and support vector identification of sequential
  minimal optimization}.
\newblock In \emph{10th NIPS Workshop on Optimization for Machine Learning},
  volume~5, 2017.

\bibitem[Sun et~al.(2019)Sun, Jeong, Nutini, and Schmidt]{sun:2019}
Y.~Sun, H.~Jeong, J.~Nutini, and M.~Schmidt.
\newblock {Are we there yet? manifold identification of gradient-related
  proximal methods}.
\newblock In \emph{The 22nd International Conference on Artificial Intelligence
  and Statistics}, pages 1110--1119, 2019.

\bibitem[Wright(1993)]{wright:1993}
S.~J. Wright.
\newblock {Identifiable surfaces in constrained optimization}.
\newblock \emph{SIAM J. Control Optim.}, 31\penalty0 (4):\penalty0 1063--1079,
  1993.

\bibitem[Wright(2012)]{wright:2012}
S.~J. Wright.
\newblock {Accelerated block-coordinate relaxation for regularized
  optimization}.
\newblock \emph{SIAM J. Optim.}, 22\penalty0 (1):\penalty0 159--186, 2012.

\end{thebibliography}

\end{document}